\documentclass[12pt,reqno]{amsart}

\numberwithin{equation}{section}

\usepackage{amssymb}
\usepackage{amsmath}
\usepackage{amsbsy}
\usepackage{amscd}
\usepackage{amsthm}
\usepackage{amsfonts}
\usepackage{enumerate}

\usepackage{color}

\usepackage{ifpdf}
\ifpdf \usepackage[pdftex,pdfstartview=FitH,pdfpagemode=none,colorlinks,bookmarks,linkcolor=blue]{hyperref} \else  \usepackage[hypertex]{hyperref} \fi

\usepackage[nobysame,abbrev,alphabetic]{amsrefs}
\newtheorem{theorem}{Theorem}[section]

\newtheorem{lemma}[theorem]{Lemma}
\newtheorem{corollary}[theorem]{Corollary}
\newtheorem{definition}[theorem]{Definition}

\newtheorem{conjecture}[theorem]{Conjecture}

\newtheorem{proposition}[theorem]{Proposition}

\newtheorem{remark}[theorem]{Remark}
\theoremstyle{definition}\newtheorem*{acknowledgments}{Acknowledgments}

\newcommand{\cC}{\mathcal{C}}

\newcommand{\cF}{\mathcal{F}}

\newcommand{\cM}{\mathcal{M}}

\newcommand{\cS}{\mathcal{S}}

\newcommand{\cW}{\mathcal{W}}

\newcommand{\bC}{\mathbb{C}}

\newcommand{\bR}{\mathbb{R}}
\newcommand{\bZ}{\mathbb{Z}}
\newcommand{\bQ}{\mathbb{Q}}

\newcommand{\bN}{\mathbb{N}}

\newcommand{\bT}{\mathbb{T}}

\newcommand{\gog}{\mathfrak{g}}
\newcommand{\goh}{\mathfrak{h}}
\newcommand{\gok}{\mathfrak{k}}

\newcommand{\gov}{\mathfrak{v}}
\newcommand{\goz}{\mathfrak{z}}

\newcommand{\bfe}{\mathbf{e}}

\newcommand{\bfm}{\mathbf{m}}
\newcommand{\bfn}{\mathbf{n}}
\newcommand{\bfr}{\mathbf{r}}

\newcommand{\bfw}{\mathbf{w}}

\newcommand{\bfz}{\mathbf{z}}

\newcommand{\Gr}{\operatorname{Gr}}

\newcommand{\GL}{\operatorname{GL}}

\newcommand{\graph}{\operatorname{graph}}

\newcommand{\rank}{\operatorname{rank}}
\newcommand{\dist}{\operatorname{dist}}

\newcommand{\End}{\operatorname{End}}

\newcommand{\Diff}{\operatorname{Diff}}

\newcommand{\di}{\mathrm{d}}
\newcommand{\Vol}{\mathrm{Vol}}

\newcommand{\CHolder}{C_{\cW_\alpha^s}^{\infty,\text{\rm H\"older}}}
\newcommand{\onto}{\xymatrix{\ar@{>>}[r]&}}
\newcommand{\da}[4]{\xymatrix{#1 \ar@<.5ex>[r]^{#2} \ar@<-.5ex>[r]_{#3} & #4}}
\newcounter{subconst}[subsection]

\newcounter{const}

\newcounter{CONST}

\begin{document}

\title[Global rigidity of higher rank Anosov algebraic actions]{Global rigidity of higher rank abelian Anosov algebraic actions}
\author[F. Rodriguez Hertz]{Federico Rodriguez Hertz}
\address{Pennsylvania State University, State College, PA 16802, USA}
\author[Z. Wang]{Zhiren Wang}
\address{Yale University, New Haven, CT 06520, USA}
\setcounter{page}{1}
\begin{abstract} We show that all $C^\infty$ Anosov $\bZ^r$-actions on tori and nilmanifolds without rank-one factor actions are, up to $C^\infty$ conjugacy, actions by automorphisms.\end{abstract}
\maketitle
{\small\tableofcontents}

\section{Introduction}

\subsection{Main result} Consider a $\bZ^r$-action $\alpha$ on a torus, a nilmanifold  or an infranilmanold $M$ by $C^\infty$ diffeomorphisms. The action by $\bfn\in\bZ^r$ is written as $\alpha^\bfn$. Following results of Franks \cite{F69} and Manning \cite{M74} and using the commutativity of the action, one can check easily that $\alpha$ is topologically conjugate to an action $\rho:\bZ^r\curvearrowright M$ by affine automorphisms. The action $\rho$ is called the {\bf linearization} of $\alpha$.

Recall a compact nilmanifold is the quotient of a simply connected nilpotent Lie group $G$ by a cocompact discrete subgroup $\Gamma$ (see \cite{M51}), and a compact infranilmanifold is a manifold that is finitely covered by a compact nilmanifold. A {\bf linear automorphism} of $G/\Gamma$ is a homeomorphism that is the projection of some $\Gamma$-preserving automorphism of $G$. An {\bf affine automorphism} of $G/\Gamma$ is a homeomorphism $f:G/\Gamma\mapsto G/\Gamma$ such that $f(gx)=f_0(g)f(x)$ for all $g\in G$, $x\in M$, where $f_0$ is an automorphism of $G$. Equivalently, an affine automorphism of $G/\Gamma$ is the composition of a linear automorphism of $G/\Gamma$ and a left translation. An affine automorphism of a compact infranilmanifold is a homeomorphism that lifts to an affine nilmanifold automorphism on a finite cover. 
We prove in this paper:

\begin{theorem}\label{Main} Suppose $\alpha:\bZ^r\curvearrowright M$ is a $C^\infty$ action on a compact infranilmanifold $M$ where $r\geq 2$, and let $\rho$ be its linearization. Assume that there exists $\bfn_0$ such that $\alpha^{\bfn_0}$ is an Anosov diffeomorphism, and that $\rho$ has no rank-one factor. Then $\alpha$ is conjugate to $\rho$ by a $C^\infty$ diffeomorphism. \end{theorem}

A rank-one factor is the projection of $\rho$ to a quotient infranilmanifold, which is, up to finite index, generated by a single $\rho^\bfn$. For details, see \S\ref{secRank}.

Even in the case of tori, the theorem is new.

A priori, it may happen that $M$ is equipped with an exotic smooth structure $\omega$ instead of the standard structure $\omega_0$ inherited from the covering Lie group $G$. $\alpha$ is smooth with respect to $\omega$, and $\rho$ is smooth with respect to $\omega_0$. In this case, the smooth conjugacy in the theorem should be understood as a diffeomorphism between $(M,\omega)$ and $(M,\omega_0)$. This in particular shows $(M,\omega)$ is actually a standard infranilmanifold and yields a contradiction. Therefore, the theorem implies:

\begin{corollary}\label{Exotic}Exotic infranilmanifolds admit no Anosov $\bZ^r$-action without rank-one factor.\end{corollary}

\subsection{Background of the problem} For a $\bZ^r$-action by diffeomorphisms on a compact manifold, there is a sharp contrast between the $r\geq 2$ case and the $r=1$ case. The simplest example for a $\bZ$-action by a single Anosov diffeomorphism is a hyperbolic toral automorphism $A$. In this case, one can modify $A$ near a fixed point to get a new diffeomorphism $f$ of the torus which is close to $A$ in the $C^\infty$-topology, but cannot be conjugate to $A$ by a diffeomorphism because the eigenvalues of its derivatives differ from those of $A$.

On the other hand, higher rank abelian Anosov actions enjoy more rigidity in this aspect. An action is said to be $C^\infty$-locally rigid if all $C^1$-perturbations are $C^\infty$-conjugate to the original action. Local rigidity was first proved for Cartan actions on tori by Katok and Lewis \cite{KL91}, and later extended to some quite general classes of actions in the works of Katok and Spatzier \cite{KS94, KS97} and Einsiedler and Fisher \cite{EF07}. Such phenomena motivated the following global rigidity conjecture by Katok and Spatzier (see \cite{G07}):
\begin{conjecture}\label{StrongConjecture} When $r\geq 2$, all ``irreducible" Anosov genuine $\bZ^r$-actions on compact manifolds are $C^\infty$-conjugate to actions on tori, nilmanifolds or infranilmanifolds by affine automorphisms.\end{conjecture}

Some progresses have been made towards the conjecture in the works of Kalinin-Spatzier \cite{KS07} and Kalinin-Sadovskaya \cites{KSad06, KSad07}. However, these results depend on strong conformality assumptions on the restriction of the action to coarse Lyapunov foliations, as well as other algebraic conditions which force the manifold to be a torus. And the conjecture is widely open.

Attention has been focused on the following natural weaker form of Conjecture \ref{StrongConjecture}:
\begin{conjecture}\label{WeakConjecture}When $r\geq 2$, all ``irreducible"\footnote{As \cite{FKS13} and Theorem \ref{Main} shows, it turns out that irreducibility is not a crucial issue here.} Anosov genuine $\bZ^r$-actions on compact infranilmanifolds are $C^\infty$-conjugate to actions by affine automorphisms.\end{conjecture}

Such global rigidity results have been obtained in two special cases, with different extra assumptions that are fairly disjoint to each other. 

In \cite{RH07}, the first author proved that Conjecture \ref{WeakConjecture} holds for a class of actions on tori, which in particular includes Cartan actions, i.e. $\bZ^r$-actions on $\bT^{r+1}$. The restrictions posed in \cite{RH07} require the rank $r$ grows linearly in terms of the dimension of the torus. 

In \cites{FKS11,FKS13}, Fisher, Kalinin and Spatzier established Conjecture \ref{WeakConjecture} under the additional  assumption that all Weyl chambers of $\rho$ contains at least one Anosov element of $\alpha$.  Another technical hypothesis, total non-symplecticity, was made in \cite{FKS11} and was later dropped in \cite{FKS13}.  They also raised the question of whether there are Anosov genuine $\bZ^r$-actions on exotic tori for $r\geq 2$, based on examples constructed by Farrel and Jones \cite{FJ78} in the $r=1$ case, and observed that the existence of such actions would contradict Conjecture \ref{WeakConjecture}.

Theorem \ref{Main} confirms Conjecture \ref{WeakConjecture}, without assuming irreducibility. This leads to a negative answer, Corollary \ref{Exotic}, to the question about actions on exotic tori,

It was remarked in \cite{FKS13} that, since there exist non-Anosov diffeomorphisms that are H\"older conjugate to Anosov toral automorphisms (constructed by Gogolev in \cite{G10}), a priori, generic elements in $\alpha$ might not be Anosov.  We will rule out this possibility, and hence verify that the hypothesis of \cite{FKS13} always holds. The proof of the fact that generic elements are Anosov will borrow ideas from both  \cite{RH07} and \cite{FKS13}, while following a new scheme.

\subsection{Organization of paper} Our strategy is to show that if a Weyl chamber of the linearized action $\rho$ contains an Anosov element, then any adjacent Weyl chamber also contains Anosov elements. This will show that there are Anosov elements in all the Weyl chambers, allowing to apply Fisher, Kalinin and Spatzier's theorem and establish Conjecture \ref{WeakConjecture}. 

In contrast to methods adopted in previous literatures (for example, \cite{MQ01}, \cite{FKS13}) when establishing smooth rigidity on nilmanifolds, our approach directly deals with individual coarse Lyapunov foliations on the nilmanifold, instead of going through an induction on the step of nilpotency which starts with the underlying torus factor of the nilmanifold.

In Section \ref{Prelim}, we state preliminary facts regarding actions on nilmanifolds and the coarse Lyapunov foliation for the linearized action $\rho$. The meaning of the assumption that $\rho$ has no rank-one factor will be explained.

When moving from an Anosov element $\alpha^{\bfn_0}$ from a Weyl chamber to a new element $\alpha^\bfn$ in a neighboring chamber, one crosses a Weyl chamber wall that corresponds to a coarse Lyapunov subgroup $V$, which in turn gives rise to a subfoliation that is contracted by the old element but expanded by the new one under the linear action. With respect to the automorphism $\rho^{\bfn_0}$, the group $G$ can be decomposed into three parts: the strong stable subgroup $G^{ss}$, $V$ and the unstable subgroup $G^u$. In Section \ref{secSmooth}, we locally parametrize the Franks-Manning conjugacy $H$, and project it into $V$ in the $G^{ss}VG^u$ decomposition  to get a coordinate function. Then we verify that this coordinate function is smooth along the stable foliation $\cW_\alpha^s$ of $\alpha^{\bfn_0}$ by solving a non-linear cohomological equation inside the Lie algebra of $V$.

In Section \ref{secAno}, it is shown that level sets of the coordinate function within $\alpha^{\bfn_0}$-stable leaves forms a subfoliation which is contracted by $\alpha^\bfn$. In order to show that this new subfoliation has smooth leaves, we check that the differential of this function along $\cW_\alpha^s$ does not degenerate. This is achieved by applying Pesin theory to an $\alpha$-invariant measure $\nu$  that is supported on the points where this condition is not satisfied. More precisely, we will rely on facts from \cite{LY85} on the almost everywhere Lipschitz continuity of the strongly stable foliation inside the stable leaves, to show that the degeneracy of the differential would cause the restriction of $H$ to the stable leaves to have a zero Jacobian, which will be proved to be impossible. Eventually, the existence of the new subfoliation will allow us to employ the quasi-Anosov criterion of Ma\~n\'e \cite{M77} to show that $\alpha^\bfn$ is also Anosov. The final lines of the proof are included in Section \ref{secConc}.

The proof of the smoothness of $h_V$ in Section \ref{secSmooth} relies on a result stating that if all partial derivatives of a H\"older function along a H\"older foliation with smooth leaves belong to the distribution space over H\"older functions, then the function is smooth along the foliation. This is a strengthening of a theorem of Rauch and Taylor \cite{RT05} and the proof is included in Appendix \ref{secSobolev}.

\begin{acknowledgments}Federico Rodriguez Hertz was supported by NSF grant DMS-1201326. Zhiren Wang was supported by NSF grant DMS-1201453 and an AMS-Simons travel grant.
\end{acknowledgments}
\section{Preliminaries}\label{Prelim}

\subsection{General settings}\label{secSetting} In this paper $\alpha$ will be a $C^\infty$-action by $\bZ^r$ on a compact nilmanifold (or more generally, a compact infranilmanifold) $M$, that is, a group morphism $\alpha:\bfn\mapsto\alpha^\bfn$ from $\bZ^r$ to the group $\Diff^\infty(M)$ of $C^\infty$ diffeomorphisms of $M$. 

We will always assume that there exists at least one $\bfn_0$ for which $\alpha^{\bfn_0}$ is Anosov, and that the linearization action $\rho:\bZ^r\curvearrowright M$, which will be defined below in \S\ref{secFM}, has no non-trivial rank-1 factor (see Definition \ref{Factor}).

Throughout the paper except in \S\ref{Exotic}, it will be assumed that the differential structure of $M$ is the one coming from the standard structure of the covering nilpotent Lie group $G$. 
Let $N$ denote the dimension  $\dim M=\dim G$.

Recall that a diffeomorphism $f$ on a Riemannian manifold $M$ is Anosov if there is a continuous splitting of the tangent bundle $TM=E_f^s\oplus E_f^u$ such that $Df$ preserves the splitting and for some $C,\lambda>0$ and every $k\geq 0$, $$\|Df^k(v)\|\leq Ce^{-\lambda k}\|v\|,\;\; \forall v\in E  _f^s;$$ $$\|Df^{-k}(v)\|\leq Ce^{-\lambda k}\|v\|,\;\; \forall v\in E_f^u.$$ In this case $M$ is everywhere foliated by the {\bf unstable foliation} $\cW_f^u$, as well as by the {\bf stable foliation} $\cW_f^s$, which are respectively tangent to $E_f^u$ and $E_f^s$. The leaves of $\cW_f^u$ and $\cW_f^s$ are given by \begin{equation}\label{AnoLeaves}\begin{split}\cW_f^u(x)=&\{y\in Y|\lim_{n\rightarrow\infty}\dist\big(f^{-n}(x),f^{-n}(y)\big)=0\};\\\cW_f^s(x)=&\{y\in Y|\lim_{n\rightarrow\infty}\dist\big(f^n(x),f^n(y)\big)=0\},\end{split}\end{equation}
and are locally immersed $C^\infty$ submanifolds. But in general these foliations are not smooth in the sense that the distribution $E_f^u$ (resp. $E_f^s$) does not necessarily vary in a $C^\infty$ way along the transversal direction to $\cW_f^u$ (resp. to $\cW_f^s$).

We will write in the rest of the paper $E_\alpha^u$, $E_\alpha^s$, $\cW_\alpha^u$ and $\cW_\alpha^s$ respectively for $E_{\alpha^{\bfn_0}}^u$, $E_{\alpha^{\bfn_0}}^s$, $\cW_{\alpha^{\bfn_0}}^u$ and $\cW_{\alpha^{\bfn_0}}^s$. 

As $\alpha$ is a commutative action, it is easy to check from \eqref{AnoLeaves} that the foliations $\cW_\alpha^s$ and $\cW_\alpha^u$ are invariant under $\alpha^\bfn$ for all $\bfn\in\bZ^r$.

\subsection{Franks-Manning conjugacy}\label{secFM}

As $\alpha^{\bfn_0}$ is an Anosov diffeomorphism of a compact nilmanifold $M=G/\Gamma$, by the classical works of Franks \cite{F69} and Manning \cite{M74}, $\alpha^{\bfn_0}$ is conjugate to an Anosov affine automorphism by a bi-H\"older homeomorphism $H:M\mapsto M$ which is homotopic to identity. By \cite{W70} any diffeomorphism of $M$ commuting with an affine automorphism is an affine automorphism, and thus $H$ conjugates the entire action $\alpha$ to an action $\rho:\bZ^r\curvearrowright M$ by affine automorphisms. To be precise, \begin{equation}\label{FMconj}\alpha^\bfn=H^{-1}\circ\rho^\bfn\circ H, \forall\bfn\in\bZ^r.\end{equation}  

By \cite[Prop. 2.4]{FKS11}, the image $\mu:=H^{-1}_*(\text{Leb})$ of the Lebesgue measure on $\bT^N$ is the unique $\alpha$-invariant absolutely continuous measure on $\bT^N$. Moreover, its density function is positive and $C^\infty$.

In fact, for each $\bfn\in\bZ^r$, $\alpha^\bfn$ induces an action on the fundamental group $\Gamma$. Every automorphism of $\Gamma$ naturally extends to a $\Gamma$-preserving automorphism of $G$, and hence induces an linear automorphism of the nilmanifold $M=G/\Gamma$. Hence $\alpha$ induces an action on $M$ by linear automorphisms. which is exactly the linear part of $\rho^\bfn$. 

By abusing notation, we will denote by $\rho$ as well the linear part of the lift of the action $\rho$ to $G$. That is, for all $g\in G$, $x\in M$, $\rho^\bfn(gx)=(\rho^\bfn g)(\rho^\bfn x)$

In addition, the differential of $\rho^\bfn$ at the identity gives a $\bZ^r$-action on the Lie algebra $\gog$ of $G$. Abusing notation again, we use $\rho$ to call this action as well. In particular, we have the relation $\exp(\rho^\bfn v)=\rho^\bfn\exp v, \forall v\in\gog$.

As $G$ is simply connected and nilpotent, the exponential map $\exp$ is a diffeomorphism between $\gog$ and $G$, with inverse $\log: G\mapsto\gog$. 

\subsection{Coarse Lyapunov decomposition} We now introduce the Lyapunov decomposition of the Lie algebra.

\begin{definition}\label{LyaSubspace} Given a $\bZ^r$-action $\rho$ on $G$ by automorphisms and a linear functional $\chi\in(\bR^r)^*$. The {\bf Lyapunov subspace}  corresponding to $\chi$ is $$\gov^\chi=\{v\in\gog\backslash\{0\}: \lim_{|\bfn|\rightarrow\infty}\frac{\log\big|\rho^\bfn v\big|-\chi(\bfn)}{|\bfn|}=0, \forall\bfn\in\bZ^r\}\cup\{0\}.$$ $\chi$ is called the {\bf Lyapunov exponent} of $\gov^\chi$.\end{definition}

\begin{lemma}\label{LyaDecomp}$\gov^\chi$ is indeed a vector subspace of $\gog$. There are only finitely many non-trivial $\gov^\chi$'s and \begin{equation}\label{LyaDecompEq}\gog=\bigoplus_\chi\gov^\chi\end{equation}\end{lemma}
\begin{proof} Fix a basis $\bfe_1,\cdots,\bfe_r$ of $\bZ^r$.  Then $\gog\otimes_\bR\bC$ is the direct sum of generalized eigenspaces $\ker (\rho^{\bfe_1}-\lambda I)^{\dim\gog}$. Since $\rho$ is a commutative action, each of these generalized eigenspaces are invariant under every $\rho^{\bfe_i}$. Hence each generalized eigenspaces splits into the direct sum of smaller generalized eigenspaces of $\rho^{\bfe_2}$, which are again invariant under all the $\rho^{\bfe_i}$'s. Keeping doing this, one can eventually decompose the entire vector space $\gog\otimes_\bR\bC$ as the direct sum of finitely many subspaces $\bigoplus_j W_j$, each of which is a common generalized eigenspace of  $\rho^{\bfe_1}, \cdots, \rho^{\bfe_r}$. On every $W_j$, the commuting matrices $\rho^{\bfe_1}, \cdots, \rho^{\bfe_r}$ can be upper triangularized together, and each $\rho^{\bfe_i}$ has only one eigenvalue $\zeta_{ij}$. Actually $\rho^{\bfe_i}|_{W_j}$ can be written as $\zeta_{ij}A_{ij}$ where $A_{ij}$ is an upper triangular matrix whose diagonal entries are equal to $1$.
It follows that for all $\bfn=\sum_{i=1}^rn_i\bfe_i$ in $\bZ^r$, $W_j$ is also a generalized eigenspace of $\rho^\bfn$ with a unique eigenvalue $\zeta_j^\bfn=\prod_{i=1}^r\zeta_{ij}^{n_i}$. And $\rho^n|_{W_j}=\zeta_j^\bfn A_j^\bfn$ where $A_j^\bfn=\prod_{i=1}^rA_{ij}^{n_i}$. One can easily check $\bfn\mapsto\zeta_j^\bfn$ and $\bfn\mapsto A_j^\bfn$ are group morphisms from $\bZ^r$, respectively to $\bC^\times$ and to $\End(W_j)$.

Furthermore, entries of $A_j^\bfn$ are polynomial in $\bfn$. Hence, for any non-zero $v\in W_j$, $\frac{\rho^\bfn v}{\zeta_j^\bfn}$ has polynomial entries as $\bfn$ varies. Therefore a non-zero vector $v\in\gog$ (or $\gog\otimes_\bR\bC$) satisfies the asymptotic formula in Definition \ref{LyaSubspace} if and only if it lies in the direct sum of the $\gog\cap W_j$'s such that the linear map $\bfn\mapsto\log|\zeta_j^\bfn|$ is equal to $\chi$. This direct sum is exact $\gov^\chi$. Moreover, $\gog=\bigoplus_j(\gog\cap W_j)=\bigoplus_\chi\gov^\chi$, and thus the number of non-trivial $\gov^\chi$'s is finite.\end{proof}

\begin{lemma}\label{LyaBracket} For all Lyapunov subspaces $\gov^{\chi_1}$ and $\gov^{\chi_2}$, the relation $[\gov^{\chi_1},\gov^{\chi_2}]\subset \gov^{\chi_1+\chi_2}$ holds.
\end{lemma}
In particular, if the Lyapunov exponent $\chi_1+\chi_2$ is absent in the Lyapunov decomposition, then $\gov^{\chi_1}$ and $\gov^{\chi_2}$ commute.

\begin{proof} Suppose $v\in \gov^{\chi_1}$, $w\in \gov^{\chi_2}$, we wish to show that $[v,w]\in \gov^{\chi_1+\chi_2}$. By decomposing into components if necessary , we can assume $v_1$ and $v_2$ are respectively in common generalized eigenspaces $W_1$ and $W_2$ constructed as in the proof of Lemma \ref{LyaDecomp}, with respective eigenvalues $\zeta_1^\bfn$ and $\zeta_2^\bfn$.  Then 
\begin{equation}
\frac{\rho^\bfn[v,w]}{\zeta_1^\bfn\zeta_2^\bfn}=\frac{\big[\rho^\bfn v,\rho^\bfn w\big]}{\zeta_1^\bfn\zeta_2^\bfn}
=\left[\frac{\rho^\bfn v}{\zeta_1^\bfn},\frac{\rho^\bfn w}{\zeta_2^\bfn}\right]
\end{equation}
has polynomial entries in term of $\bfn$. But $|\zeta_j^\bfn|=e^{\chi_j(\bfn)}$ for $j=1,2$. Thus the vector $[v,w]$, if it doesn't vanish,  satisfies $\lim_{|\bfn|\rightarrow\infty}\frac{\log\big|\rho^\bfn [v,w]\big|-\chi(\bfn)}{|\bfn|}=0$ with $\chi=\chi_1+\chi_2$. This completes the proof.
\end{proof}

\begin{definition}\label{CoarseDef}For a Lyapunov exponent $\chi$ that is present in the Lyapunov decomposition, the corresponding {\bf coarse Lyapunov subspace} is
$$\gov^{[\chi]}=\bigoplus_{\chi'=c\chi, c>0}\gov^{\chi'}.$$\end{definition}

So a coarse Lyapunov subspace is a direct sum of all Lyapunov subspaces whose exponents are positively proportional in $(\bR^r)^*$. When two vectors belong to different coarse Lyapunov subspaces, one can find an $\bfn$ such that $\rho^{i\bfn}$ expands one of the vectors but contracts the other as $i\rightarrow\infty$. However vectors from the same coarse Lyapunov subspace cannot be distinguished in this way.

This gives the coarse Lyapunov decomposition  \begin{equation}\label{CoarseDecomp}\gog=\bigoplus \gov^{[\chi]}\end{equation}

Both the Lyapunov decomposition \eqref{LyaDecomp} and the coarse Lyapunov decomposition \eqref{CoarseDecomp} are $\rho$-invariant.

One consequence Lemma \ref{LyaBracket} is that, as the Lyapunov exponents in each coarse Lyapunov subspace $\gov^{[\chi]}$ are positively proportional to each other and hence closed under addition, $\gov^{[\chi]}$ is a Lie subalgebra. Therefore we have:

\begin{definition} To each coarse Lyapunov subspace $\gov^{[\chi]}$ is associated a closed connected subgroup $V^{[\chi]}=\exp\gov^{[\chi]}\subset G$, which will be called a {\bf coarse Lyapunov subgroup}.\end{definition}

Fixing $\bfn\in\bZ^r$, the direct sum of the $\gov^{\chi}$'s with $\chi(\bfn)<0$ (resp. $>0$) is the stable (resp. neutral, unstable) Lie subalgebra of $\rho^\bfn$ and is denoted by $\gog_{\rho^\bfn}^s$ (resp. $\gog_{\rho^\bfn}^u)$. The corresponding Lie subgroups are denoted respectively by $G_{\rho^\bfn}^s$ and $G_{\rho^\bfn}^u$.

Each coarse Lyapunov subspace $\gov^{[\chi]}$ gives rise to the {\bf Weyl chamber wall} $\ker \chi\subset \bR^r$, which is a hyperplane that divides $\bR^r$ into the {\rm positive (resp. negative) Lyapunov half space} \begin{equation}L^{[\chi],+}=\{\bfw\in\bR^r|\chi(\bfw)>0\}\  (\text{resp. } L^{[\chi],-}=\{\bfw\in\bR^r|\chi(\bfw)<0\}).\end{equation}

\begin{remark}\label{Symplec} The definitions of $\ker\chi$, $L^{[\chi],+}$ and $L^{[\chi],-}$ clearly do not depend on the choice of $\chi$ representing $[\chi]$, and two different coarse Lyapunov subspaces $\gov^{[\chi]}$ and $\gov^{[\chi']}$ give rise to the same Weyl chamber wall if and only if the exponents $\chi$ and $\chi'$ are negatively proportional, in which case $L^{[\chi],+}=L^{[\chi'],-}$ and $L^{[\chi],-}=L^{[\chi'],+}$.\end{remark}

\begin{definition}\label{WeylChamber}The connected components of $\bR^r\backslash\bigcup_{[\chi]}\ker\chi$ are called {\bf Weyl chambers} of $\rho$.\end{definition}

By construction, Weyl chambers are the minimal non-trivial intersections of different Lyapunov half spaces.

\subsection{Correspondence between foliations} As earlier, $\alpha$ is a $\bZ^r$-action on $M$, with Anosov element  $\alpha^\bfn_0$ and linearization $\rho$.  For simplicity, let $\gog^s$, $\gog^u$, $G^s$ and $G^u$ denote respectively  $\gog_{\rho^{\bfn_0}}^s$, $\gog_{\rho^{\bfn_0}}^u$, $G_{\rho^{\bfn_0}}^s$ and $G_{\rho^{\bfn_0}}^u$. It is easy to check that for $x\in M$,  $G^sx$,  $G^ux$ are just the  stable and unstable foliations $\cW_\rho^s(x)$ and $\cW_\rho^u(x)$ for the affine automorphism $\rho^{\bfn_0}$. Since $\alpha$ and $\rho$ are conjugate by $H$, we have:

\begin{equation}\label{CorreUS}H(\cW_\alpha^s(x))=G^sH(x),\;\;\;\; H(\cW_\alpha^u(x))=G^uH(x).\end{equation}

In particular, for $\square\in\{s,u\}$, $G^\square x$ is a manifold with dimension $\dim\cW_\alpha^\square$ and thus $\dim\gog^\square=\dim\cW_\alpha^\square=\dim E_\alpha^\square$. It follows that $\dim\gog^s+\dim\gog^u=N=\dim\gog$. Since  $\gog^s$ and $\gog^u$ have trivial intersection by construction, $\gog=\gog^s\oplus\gog^u$. In other words, $\rho^{\bfn_0}$ is a hyperbolic automorphism, or equivalently, for all Lyapunov exponent $\chi$ for $\rho$, $\chi(\bfn_0)\neq 0$.

This implies there is no trivial exponent in the Lyapunov decomposition \eqref{LyaDecompEq} and justifies Definition \ref{WeylChamber}.

\subsection{Rank assumption and ergodicity of generic elements}\label{secRank}

We will assume $\rho:\bZ^r\curvearrowright M$ has no rank-one factor in the sense below following previous literatures, e.g. \cites{KK99,KKS02}.

Let $M=G/\Gamma$ be a nilmanifold. Suppose there is a surjective morphism $\pi: G\mapsto\breve G$ to a nontrivial nilpotent Lie group $\breve G$ of lower dimension, such that $\Gamma$ is projected to another cocompact discrete subgroup $\breve\Gamma\subset\breve G$. In this case $\pi$ defines a projection, which is still denoted by $\pi$, from $M$ to $\breve M=\breve G/\breve \Gamma$. We say the new nilmanifold $\breve M$ is an {\bf algebraic factor} of $M$. 

Furthermore, suppose $\rho$ is a $\bZ^r$-action on $M$ by affine automorphisms. If the linear part of $\rho$, which is an action by group automorphisms of $G$ and is denoted also by $\rho$, preserves the kernel of the projection $\pi$, then $\rho$ descends to an action on $\breve M$ by automorphisms, which is a factor action of $\rho$. 

\begin{definition}\label{Factor} A $\bZ^r$-action $\rho$ on a nilmanifold $M$ by automorphisms is said to have a {\bf rank-one factor} if $\rho$ admits a factor action $\breve\rho$ on a nontrivial algebraic factor $\breve M$ of $M$, and for some finite-index subgroup  $\Lambda\subset\bZ^r$,  the linearization $\{(\breve\rho)^\bfn:\bfn\in\Lambda\}$ consists of a cyclic group of affine automorphisms. \end{definition}

A $k$-step nilmanifold $M=G/\Gamma$ arises as principal bundles. The center $G'= [G,G]$ is a $(k-1)$-step nilpotent normal subgroup of $G$, and both $M'=G'/(G'\cap\Gamma)$ and $M_0=(G/G')/\big(\Gamma/(G'\cap\Gamma)\big)$ are compact. $M$ is a bundle over $M_0$ with $M'$ fibres. $M'$ is a $(k-1)$-step nilmanifold. $M_0$ is a torus as $G/G'$ is abelian, and is naturally an algebraic factor of $M$, called the maximal torus factor of $M$. Since $G'$ is proper in $G$, $\dim M_0>0$ as long as $\dim M>0$.

Since automorphisms of $G$ preserve the center subgroup $G'$, any action $\rho$ by automorphisms projects to $M_0$.

By \cite{S99}, having no non-trivial factor is actually equivalent to the assumption in \cite{FKS13} that there is a $\bZ^2$-subaction whose all non-trivial elements act ergodically. In fact, we have:

\begin{lemma}\label{rank2}Let $\rho:\bZ^r\curvearrowright M$ be an action by nilmanifold automorphisms. Then the following are equivalent:
\begin{enumerate}
\item $\rho$ has no rank-one factor;
\item There is a subgroup $\Sigma\subset\bZ^r$ such that $\Sigma\cong\bZ^2$ and $\rho^\bfn$ is an ergodic toral automorphism for all $\bfn\in\Sigma\backslash\{\bf0\}$.
\item There are finitely many subgroups $A_1,\cdots,A_k\subset \bZ^r$, all of rank at most $r-2$, whose union contains every $\bfn\in\bZ^r$ that fails to act ergodically by $\rho$;
\end{enumerate}
\end{lemma}

\begin{proof}{\bf Case of tori.} We first prove the lemma assuming $M$ is a torus $\bT^N$. 

The equivalence between (1) and (2) was proved by Starkov in \cite{S99}. In that paper, this equivalence was stated for actions by linear automorphisms of $\bT^N$. But it is not hard to check that if $\rho$ is an action by affine toral automorphisms, then, after relabeling a new point as the origin of $\bT^N$ if necessary, the restriction of $\rho$ to a finite-index subgroup $\Lambda$ of $\bZ^r$ consists of linear automorphisms. This fact allows to easily pass from the linear case to the case of general affine actions.

(2)$\Rightarrow$(3): As all the $\rho^\bfn$'s, regarded as elements of $\GL(N,\bZ)$, are commuting matrices, there is a basis in $\bC^N$ that upper-triangularizes them simultaneously. Denote by $\zeta_i^\bfn$ the $i$-th diagonal entry of $\rho^\bfn$ in this basis. Then $\zeta_i^\bfn:\bZ^r\mapsto\bC^\times$ is a group morphism.

It is well known that a toral automorphism is not ergodic if and only if at least one eigenvalue is a root of unity. Now fix $1\leq i\leq N$ and look at $A_i:=\{\bfn\in\bZ^r:\zeta_i^\bfn\text{ is a root of unity}\}$, then $A_i$ is a subgroup of $\bZ^r$ and it suffices to show that $\rank(A_i)\leq r-2$. Suppose otherwise, then any rank-2 subgroup of $\bZ^r$ has a non-trivial intersection with $A_i$, hence contains non-zero elements that acts non-ergodically under $\rho$, which contradicts (2).

(3)$\Rightarrow$(2): let $U_i\subset\bR^r$ be the subspace spanned by $A_i$, then since $\dim U_i=\rank A_i\leq r-2$, $Y_i:=\{P\in\Gr(2,r):\dim (P\cap U_i)>0\}$ is a proper subvariety in the Grassmannian $\Gr(2,r)$ of two dimensional planes in $\bR^r$. As rational subspaces form a dense subset of $\Gr(2,r)$, a generically positioned rational plane $P$ does not belong to any of the $Y_i$'s. Note $\Sigma:=P\cap\bZ^r$ is isomorphic to $\bZ^2$ by rationality, and $\Sigma\cap A_i\subset P\cap U_i=\{\bf0\}$. Hence $\Sigma$ satisfies the requirements in (2).

\noindent{\bf General nilmanifolds.} As we already treated the case of tori, it suffices to show each of the assertions (1)-(3) for a nilmanifold $M$ is equivalent to the corresponding claim for the induced action on the maximal torus factor $M_0$ of $M$. 

For (2) and (3), this directly follows from Parry's theorem \cite{P69} that an automorphism on $M$ is ergodic if and only if it induces an ergodic automorphism of $M_0$. So we only need to show the action $\rho:\bZ^r\curvearrowright M$ has a rank-one factor if and only if the induced action on $M_0$ does. The ``if" part is obvious since factors of $M_0$, together with the induced actions on them, descends from $M$ through $M_0$. Assume now $\rho$ has a rank-one factor, or equivalently, for some surjective morphism $G\mapsto G/L$ where $L$ is a closed normal proper subgroup of $G$ and is invariant under the lifting of the action $\rho$ to $G$, $\Gamma/(L\cap\Gamma)$ is discrete in $G/L$ and the induced action on the nilmanifold $Y=(G/L)/\big(\Gamma/(L\cap\Gamma)\big)$ is cyclic up to finite index. Consider the maximal torus factor $Y_0$ of $Y$, which is of positive dimension. It is the quotient of the Lie group \begin{equation}\begin{split}(G/L)/[G/L,G/L]=&(G/L)/\big(G'/(G'\cap L)\big)=G/G'L\\
=&(G/G')/\big(L/(G'\cap L)\big)\end{split}\end{equation} by the natural projection of $\Gamma$ in it,  where $G'=[G,G]$. Thus $Y_0$ is an algebraic factor of $M_0=(G/G')/\big(\Gamma/(G'\cap\Gamma)\big)$. Since the lifted action of $\rho$ on $G$ preserves $G'$, $L$,  and $\Gamma$, $\rho$ factors onto an action on $Y_0$ through $M_0$. This action, which also descends from that on $Y$, must be cyclic up to finite index. In other words, $\rho:\bZ^r\curvearrowright M_0$ has a rank-one factor. \end{proof}

The proof of the implication (3)$\Rightarrow$(2) may be restated as:

\begin{corollary}\label{GenePlane} If $\rho$ has no rank-one factor, then in the Grassmannian $\Gr(2,r)$ of $2$-dimensional subspaces in $\bR^r$, there is an open dense subset $\Omega$ such that for any rational subspace $P\in \Omega$ and any non-zero element $\bfn\in P\cap\bZ^r$, $\rho^\bfn$ is ergodic.
\end{corollary}

{\quote

From now on we will adopt notations developed in Section \ref{Prelim} and work under the following assumptions unless otherwise remarked:\it 
\begin{itemize}
\item  $M$ is a compact nilmanifold, with standard differential structure;\advance\rightskip 1cm
\item $\alpha^{\bfn_0}$ is an Anosov diffeomorphism;
\item The linearization $\rho:\bZ^r\curvearrowright M$ of $\alpha$ has no rank-one factor. \advance\rightskip 1cm
\end{itemize}
}

\section{Smooth conjugacy in certain coarse Lyapunov subspaces}\label{secSmooth}

Since the Franks-Manning conjugacy $H$ is H\"older continuous homotopic to identity, there is a H\"older continuous function $h:M\mapsto G$ such that $H(x)=h(x)x$ for all $x\in M$. The conjugacy map $H$ is $C^\infty$ if the map $h:M\mapsto G$ is $C^\infty$. While it is hard to show $h$ is $C^\infty$ at once we will make a group decomposition of $G$ into coarse Lyapunov subgroups $V$'s, and study the $V$-component $h_V$ of $h$ in the decomposition. It will be shown that for certain Lyapunov subgroup $V$, the restrictions of $h_V$ to stable manifolds of $\alpha^{\bfn_0}$ are $C^\infty$.

\subsection{The strong stable -- weak stable -- unstable decomposition}\label{secSVU}

We have already remarked that $\chi(\bfn_0)$ does not vanish for any Lyapunov exponent $\chi$, that is, $\bfn_0$ lies in non of the Weyl chamber walls, and thus it is in the interior of one Weyl chamber $\cC_0$. 

From now on, we fix a Weyl chamber $\cC$ that is adjacent to $\cC_0$, and let $\chi$ be a Lyapunov exponent from the Lyapunov decomposition such that $\ker\chi$ be the Weyl chamber wall between $\cC$ and $\cC_0$. Denote by $\gov$ the coarse Lyapunov subspace $\gov^{[\chi]}$, and let $V=V^{[\chi]}=\exp\gov$. 

Note that $-\bfn_0$ is in the opposite Weyl chamber $-\cC_0$ and $\alpha^{-\bfn_0}$ is also Anosov. Therefore, without loss of generality, we may assume $V\subset G^s$ or equivalently $\cC_0\subset L^{[\chi],-}$. This is because otherwise $V\subset G^u$ is stable under $-\bfn_0$, and we can always study $-\bfn_0$ and $-\cC_0$ instead.

We define the strong stable subspace of $\rho^{\bfn_0}$ by $\gog^{ss}=\bigoplus_{\substack{\gov^{[\chi']}\neq \gov\\ \gov^{[\chi']}\subset\gog^s}}\gov^{[\chi']}$, which is the complement of $\gov$ in $\gog^s$ in the coarse Lyapunov decomposition.

\begin{lemma}\label{SUsubalg}\begin{enumerate}
\item $\gog^{ss}$ is a Lie subalgebra;
\item $\gov\oplus\gog^u$ is a Lie subalgebra;
\item $[\gov,\gog^{ss}]\subset\gog^{ss}$;

\end{enumerate}\end{lemma}
\begin{proof} Take an element $\bfz$ from the open cone inside $\ker\chi$ that touches both $\cC_0$ and $\cC$. Then $\chi(\bfz)=0$. Furthermore, for another coarse Lyapunov subspace $\gov^{[\chi']}$ where $\chi'$ is not negatively proportional to $\chi$, $\ker\chi'$ does not coincide with $\ker\chi$ by Remark \ref{Symplec}; so $\bfz$ and $\cC_0$ are on the same side of $\ker\chi'$ and in particular $\chi'(\bfz)$ and $\chi'(\bfn_0)$ has the same sign.  

For every $\gov^{[\chi']}\subset\gog^{ss}$, because $\chi'(\bfn_0)$ and $\chi(\bfn_0)$ are both negative, $\chi'$ is not negatively proportional to $\chi$. Hence $\chi'(\bfz)$ is negative as $\chi'(\bfn_0)$ is.

For $\gov^{[\chi']}\subset\gog^u$, it is possible that $\chi'$ is negatively proportional to $\chi$. As $\chi'(\bfn_0)>0$, it follows $\chi'(\bfz)>0$ unless $\gov^{[\chi']}=\gov^{[-\chi]}$.

Therefore, \begin{equation}\label{SUsubalgEq}\gog^{ss}=\oplus_{\chi'(\bfz)<0}\gov^{\chi'}, \gog^u=\gov^{[-\chi]}\oplus\bigoplus_{\chi'(\bfz)>0}\gov^{\chi'},\end{equation}
where $\gov^{[-\chi]}$ may be trivial. 

Notice that $\gov\oplus\gov^{[-\chi]}=\bigoplus_{\chi'(\bfz)=0}\gov^{\chi'}$. In consequence, $\gov\oplus\gog^u=\bigoplus_{\chi'(\bfz)\geq 0}\gov^{\chi'}$.

Together with Lemma \ref{LyaBracket}, these formulae imply the lemma. \end{proof}

The strong stable subgroup with respect to $\rho^{\bfn_0}$ is $G^{ss}=\exp\gog^{ss}$. Since $\gog^{ss}\oplus\gov\oplus\gog^u=\gog^s\oplus\gog^u=\gog$ and all the components are Lie subalgebras, $G$ uniquely splits as the product $G^{ss}\cdot V\cdot G^u$.

\begin{lemma} Suppose $H$, $K$ are connected close subgroups of a nilpotent Lie group $G$ and $\goh, \gok\subset\gog$ are the corresponding Lie algebras. If $\goh\oplus\gok=\gog$ then the multiplication map $(a,b)\mapsto ab$ is a $C^\infty$ diffeomorphism from $H\times K$ to $G$.\end{lemma}
\begin{proof}The smoothness and injectivity of the map are easy to see. For surjectivity, see \cite{M78}. It remains to confirm that the inverse of the multiplication map is also smooth. This is equivalent to that the differential of the multiplication map is non-degenerate everywhere. Making use of group translation, one only needs to look at the differential at the identity $(e,e)$, which is the isomorphism $(\di a,\di b)\mapsto \di a+\di b$ between $\goh\oplus\gok$ and $\gog$.\end{proof}
 
\begin{corollary}\label{SVU}\begin{enumerate}\item The multiplication map $(a,b,c)\mapsto abc$
from $G^{ss}\times V\times G^u$ to $G$ is a $C^\infty$ diffeomorphism;
\item $(a,b)\mapsto ab$ is a $C^\infty$ diffeomorphism from $G^{ss}\times V$ to $G^s$.\end{enumerate}\end{corollary}

This follows immediately from the lemma. 

\begin{definition}Given an element $g\in G$, its unique $G^{ss}VG^u$ decomposition will be written as $g_{ss}g_Vg_u$, and $g_s=g_{ss}g_V$ will stand for the $G^s$ part.\end{definition}

\begin{corollary}\label{UStoSU}For any given $b\in G$, the restriction of the map $a\mapsto (ab)_u$ to $G^u$ is a $C^\infty$ diffeomorphism from $G^u$ to itself. \end{corollary}
\begin{proof}The map $a\mapsto(ab)_u$ is smooth by Corollary \ref{SVU}. Note the equation $ab=(ab)_s(ab)_u$ also writes $(ab)_s^{-1}a=(ab)_ub^{-1}$. Hence if $a\in U$ then $a=\big((ab)_ub^{-1}\big)_u$, which gives a smooth inverse of the original map. This proves the corollary.\end{proof}

\begin{corollary}\label{Vpart}\begin{enumerate}\item$(ab)_V=a_V(a_ub)_V$.
\item \label{Vpart2} If $a\in G^s$, then $(ab)_V=a_Vb_V$.\end{enumerate}\end{corollary}
\begin{proof} (1) is proved by the following computation:
\begin{align*}
ab=&a_{ss}a_Va_ub=a_{ss}a_V(a_ub)_{ss}(a_ub)_V(a_ub)_u\\
=&a_{ss}a_V(a_ub)_{ss}a_V^{-1}\cdot a_V(a_ub)_V\cdot (a_ub)_u
\end{align*}
Note that $a_V(a_ub)_V\in V$ and $a_{ss}a_V(a_ub)_{ss}a_V^{-1}\in G^{ss}$, using the fact that $V$ normalizes $G^{ss}$, a consequence to Lemma \ref{SUsubalg}.(2).

(2) is an immediate consequence to (1). 
\end{proof}

$h(x)$ splits uniquely as $h_{ss}(x)h_V(x)h_U(x)$ by Corollary \ref{SVU}. Write $h_s=h_{ss}\cdot h_V$. The functions $h_s$, $h_{ss}$, $h_V$ and $h_U$ are also H\"older continuous.

We will show that $h_V$ is smooth along $\cW_\alpha^s$. The first step of doing this is to deduce an equation that characterizes $h_V$. 

\subsection{The cohomological equation} 

Given $\bfn\in\bZ^r$, the choice of which will be fixed later, we have that $\alpha^\bfn$ is homotopic to $\rho^\bfn$. Thus the map $\rho^{-\bfn}\circ\alpha^\bfn$ is homotopic to identity and can be written as $Q_1(x)x$ for some $C^\infty$ function $Q_1: M\mapsto G$.

By the construction of $h$, for any $x\in M=G/\Gamma$, 
\begin{equation}\label{CohoEq1}\begin{split}h(x)x=&H(x)=\rho^{-\bfn}H(\alpha^\bfn x)\\
=&\rho^{-\bfn}\big(h(\alpha^\bfn x)\cdot (\alpha^\bfn x)\big)\\
=&\big(\rho^{-\bfn}h(\alpha^\bfn x)\big)\cdot\big(\rho^{-\bfn}(\alpha^\bfn x)\big)\\
=&\big(\rho^{-\bfn}h(\alpha^\bfn x)\big)Q_1(x)\cdot x.
\end{split}\end{equation}

\begin{lemma}\label{CohoEquation}There exists a $C^\infty$ map $Q:M\mapsto G$, such that 
\begin{equation}\label{CohoEq}h=(\rho^{-\bfn}h\circ\alpha^\bfn)\cdot Q.\end{equation}\end{lemma}

\begin{proof} We show the following claim first:

{\it If a continuous map $f: M\mapsto G$ satisfies $f(x)x=x$ for all $x\in M$, then there exists $\gamma_0\in\Gamma\cap\mathrm{Center}(G)$ such that $f(x)=\gamma_0$ for all $x$} 

Lift $f$ from $G/\Gamma$ to $G$ and regard it as a continuous function on $G$ with $f(g)=f(g\gamma)$ for all $\gamma\in\Gamma$. Then for all $g\in G$, $f(g)g\Gamma=g\Gamma$, or equivalently, $g^{-1}f(g)g\in\Gamma$. Since $\Gamma$ is discrete and $g^{-1}f(g)g$ is continuous, the value must be a constant $\gamma_0\in\Gamma$. Hence $f(g)=g\gamma_0g^{-1}$. However, because $f(g)=f(g\gamma)$ for all $g\in G$, $\gamma\in\Gamma$, the element $\gamma_0$ has to commute with all $\gamma\in\Gamma$. As $\Gamma$ is a Zariski-dense subgroup in $G$, this forces $\gamma_0$ to be in the center of $G$. Hence for all $g$, $f(g)=g\gamma_0g^{-1}=\gamma_0$.

By multiplying both sides by $\big(h(x)\big)^{-1}$ in \eqref{CohoEq1}, we see that $$h^{-1}\cdot(\rho^{-\bfn}h\circ\alpha^\bfn)\cdot Q_1=\gamma_0$$ for some such $\gamma_0$. It suffices to take $Q=Q_1\cdot\gamma_0^{-1}$, which is $C^\infty$.
\end{proof}

Using Corollary \ref{Vpart}, the cohomological equation \eqref{CohoEq} can be projected to the subgroup $V$:
\begin{equation}\label{CohoEqV}h_V=(\rho^{-\bfn}h_V\circ\alpha^\bfn)\cdot\Big((\rho^{-\bfn}h_u\circ\alpha^\bfn)\cdot Q\Big)_V.\end{equation}

Note that the smooth function $Q$ depends on $\bfn$. We now make a choice of $\bfn$ that will work better later.

\begin{lemma}\label{ExpaChoice} For all $\eta>0$, there exists $\bfn\in\bZ^r$ such that:
\begin{enumerate}
\item $\bfn$ belongs to a subgroup $\Sigma\cong\bZ^2$ of $\bZ^r$ such that $\rho^\bfm$ is an ergodic nilmanifold automorphism for all $\bfm\in\Sigma\backslash\{0\}$,
\item For all $\gov^{\chi'}\subset \gov$, $\chi'(\bfn)>0$;
\item For all $\gov^{\chi'}\subset \gog^s$, $\chi'(\bfn)\leq\eta|\bfn|$;
\item $\log\|D\alpha^\bfn|_{E_\alpha^s}(x)\|\leq\eta|\bfn|$ for all $x\in M$.
\end{enumerate}
\end{lemma}

\begin{proof} We first construct an $\bfn$ that satisfies conditions (1)-(3).

Let $\chi_+$ denote the largest element in norm from $[\chi]$. Since all  Lyapunov exponents in $\gov^{[\chi]}$ are positively proportional, each of them writes $\chi'=a\chi_+$ for some $0< a\leq 1$ that depends on $\chi'$.

If $\bfn$ is chosen from $L^{[\chi],+}$ such that the angle between $\bfn$ and $\ker\chi$ is sufficiently small, then $\frac{\chi_+(\bfn)}{|\bfn|}$ is positive and bounded by $\eta$. Then
\begin{equation}\label{ExpaChoiceEq}0<\chi'(\bfn)\leq\eta|\bfn|, \forall\gov^{\chi'}\subset\gov^{[\chi]}.\end{equation}

Recall $\ker\chi$ is one of the Weyl chamber walls that bounds $\cC_0$ as well as $-\cC_0$. As the collection of rational subspaces is dense in the Grassmannian, we can fix a generic two dimensional subspace $P\subset\bZ^r$, in the sense that $P$ is in the open dense subset $\Omega$ in Corollary \ref{GenePlane}, such that $P$ intersects the interior of $\cC_0$ as well as the interior of the open cone in $\ker\chi$ that touches $\cC_0$. Notice that $P$ also intersects the interior of $-\cC_0$ and that of $\ker\chi\cap\partial(-\cC_0)$. Let $\Sigma=P\cap\bZ^r$, then $\Sigma\cong\bZ^2$ is a lattice in $P$ and $\rho^\bfm$ is ergodic for all $\bfm\in\Sigma\backslash\{\bf0\}$. $\bfn$ will be chosen from $\Sigma$ so that (1) is satisfied.

We start from $\cC_0$, cross the wall $\ker\chi$ and get $\bfn$. More precisely, we pick $\bfn\in\Sigma$ such that it lies in the adjacent Weyl chamber $\cC$ that shares the wall $\ker\chi$ with $\cC_0$, and that the angle between $\bfn$ and $\ker\chi$ is sufficiently small. Then $\bfn\in L^{[\chi],+}$ since $\cC_0\subset L^{[\chi],-}$, and thus condition (2) is verified.

On the other hand, when the angle is small enough, because no Weyl chamber wall other than $\ker\chi$ separates $\bfn_0$ from $\bfn$, by Remark \ref{Symplec}, for any Lyapunov exponent $\chi'$, as long as $\chi'$ is not in $[\chi]$ or $[-\chi]$, $\chi'(\bfn)$ has the same sign as $\chi'(\bfn_0)$. Since $\gog^s=\gov\oplus\gog^{ss}$ and $\gov^{[-\chi]}\subset\gog^u$, to verify condition (3) we only need that $\chi'(\bfn)\leq\eta|\bfn|$ for every $\chi'\in[\chi]$, which follows from \eqref{ExpaChoiceEq}.

Next, we make (4) satisfied by replacing $\bfn$ with a positive multiple of itself if necessary. Notice that doing this will not affect conditions (1)-(3). 

Let us assume both the Franks-Manning conjugacy $H$ and $H^{-1}$ are $\gamma$-H\"older with $0<\gamma<1$. Then thanks to condition (3), by applying the following Lemma \ref{Subadd} to the conjugacy $H$ between the maps $\alpha^\bfn$ and $\rho^\bfn$, along the foliations $\cW_\alpha^s$ and $E_\rho$, we find $k\in\bN$ such that $\log\|D\alpha^{k\bfn}|_{E_\alpha^s}(x)\|<\frac{2\eta}\gamma|k\bfn|$ for all $x$. So if we have worked with $\frac{\gamma\eta}2<\eta$ in place of $\eta$ from the beginning, we would obtain $\bfn$ satisfying conditions (1)-(3) and some $k\in\bN$ such that $\log\|D\alpha^{k\bfn}|_{E_\alpha^s}(x)\|<\eta|k\bfn|$ for all $x\in M$. It suffices to choose $k\bfn$ instead of $\bfn$.\end{proof}

We shall prove now two Lemmas relating the contraction and expansion rates of two H\"older conjugate diffeomorphisms. The first Lemma (that we need in the proof of the previous Lemma \ref{ExpaChoice}) deals with the non-expanding case, the second one needed in the future deals with the contracting case.

\begin{lemma}\label{Subadd}Suppose two diffeomorphisms $f_0,f_1:M\mapsto M$ of a smooth manifold $M$ are conjugated to each other by a $\gamma$-H\"older homeomorphism $\phi$, where $0<\gamma<1$, $\phi\circ f_0=f_1\circ\phi$. Suppose in addition that, $\cF_i$ is an $f_i$-invariant continuous foliation with smooth leaves for $i=0,1$, and $\phi$ sends $\cF_0$ to $\cF_1$. Let $F_i(x)\subset T_xM$ be the distributions consisting of tangent spaces to $\cF_i$.  Assume $\log\|Df_1|_{F_1}\|$ is uniformly bounded by some $\beta\geq 0$ at all points, then for all $\epsilon>0$ there exists $k\in\bN$ such that $\log\|Df_0^k|_{F_0}\|\leq(\frac\beta\gamma+\epsilon)k$.
\end{lemma}

\begin{proof}Suppose $x,y$ are two nearby points in the same $\cF_0$-leaf and $k\in\bN$, then $\phi(x)$ and $\phi(y)$ are in the same $\cF_1$-leaf. Hence \begin{equation}\label{SubaddEq1}\begin{split}
\dist(f_0^{-k}x, f_0^{-k}y)\geq &\frac1C\dist\big(\phi(f_0^{-k}x),\phi(f_0^{-k}y)\big)^{\frac1\gamma}\\
\geq&\frac1C\dist\big(f_1^{-k}\phi(x), f_1^{-k}\phi(y)\big)^{\frac1\gamma}.
\end{split}\end{equation}
Since $\log\|Df_1|_{F_1}\|\leq\beta$, $\dist\big(f_1^{-k}\phi(x), f_1^{-k}\phi(y)\big)$ cannot decay to 0 at a rate faster than $O(e^{-(\beta+\frac{\gamma\epsilon}2)k})$. It follows from \eqref{SubaddEq1} that $\dist(f_0^{-k}x, f_0^{-k}y)$ decays, if at all, at a rate slower than $O(e^{-\frac1\gamma(\beta+\frac{\gamma\epsilon}2)k})$. Therefore with respect to any ergodic $f$-invariant measure probability $\nu$, all Lyapunov exponents of $f$ restricted to the invariant distribution $F$ tangent to $\cF_0$, are bounded from above by $\frac1\gamma(\beta+\frac{\gamma\epsilon}2)=\frac\beta\gamma+\frac\epsilon2$. 

Therefore, for all ergodic $\nu$ and for $\nu$-a.e. $x$, \begin{equation}\label{SubaddEq3}\overline\lim_{k\rightarrow\infty}\frac1k\log\|Df_0^k|_{F_0}(x)\|\leq\frac\beta\gamma+\frac\epsilon2.\end{equation}

Let $a_k(x)=\log\|D f_0^{k}|_{F_0}(x)\|-(\frac\beta\gamma+\epsilon)k$, then $\overline\lim_{k\rightarrow\infty}\frac1k a_k(x)\leq-\frac\epsilon2$ for $\nu$-a.e. $x$. By Fatou's Lemma, $\overline\lim_{k\mapsto\infty}\frac1k\int a_k\di\nu\leq  -\frac\epsilon2<0$.

Moreover, one can easily check that for all $x$ and for all $k,l\in\bN$, $$a_{k+l}(x)\leq a_k( f_1^{l\bfn}x)+a_l(x)$$ and $$a_k(x)\leq a_k( f_1^{l\bfn}x)+a_l(x)+b_l( f_1^{k}.x)$$ where $b_l(x)=\log\|D f_1^{-l\bfn}( f_1^{l\bfn}.x)\|+\epsilon l$. 

By applying \cite[Prop. 3.4]{RH07} to the sequence of continuous functions $a_k(x)$, one concludes that there is a $k$ such that $a_k(x)<0$ for all $x\in M$. That is $\log\|D f_0^{k}|_{F}(x)\|<(\frac\beta\gamma+\epsilon)k$, which is the content of the lemma.
\end{proof}

The proof of the next Lemma is essentially the same as the previous one so we omit its proof. 

\begin{lemma}\label{Subadd2} Suppose two diffeomorphisms $f_0,f_1:M\mapsto M$ of a smooth manifold $M$ are conjugated to each other by a homeomorphism $\phi$,  $\phi\circ f_0=f_1\circ\phi$ such that its inverse $\phi^{-1}$ is $\gamma$-H\"older where $0<\gamma<1$. Suppose in addition that, $\cF_i$ is an $f_i$-invariant continuous foliation with smooth leaves for $i=0,1$, and $\phi$ sends $\cF_0$ to $\cF_1$. Let $F_i(x)\subset T_xM$ be the distributions consisting of tangent spaces to $\cF_i$.  Assume $\log\|Df_1|_{F_1}\|$ is uniformly bounded by some $\beta<0$ at all points, then for all $\epsilon>0$ there exists $k\in\bN$ such that $\log\|Df_0^k|_{F_0}\|\leq(\beta\gamma+\epsilon)k$.
\end{lemma}

\subsection{Smoothness of parameters}\label{sechuSmooth}

Using the cohomological equation \eqref{CohoEqV}, it will be shown that $h_V$ is smooth along the $\cW_\alpha^s$ leaves, with H\"older continuous partial derivatives of all orders. The parameters in the equation involve $h_u$ and the smooth function $Q$.  We verify first $h_u$ has the same smoothness as desired for $h_V$.

Denote by $\partial_{\cW_\alpha^s}^k\phi$ the vector consisting of all partial derivatives up to order $k$ of a function, or a distribution $\phi$ on $M$, along the foliation $\cW_\alpha^s$. This definition makes sense because of the following regularity properties of $\cW_\alpha^s$:

For any point $x\in M$, there is a local chart map $\Gamma: \Omega_1\times \Omega_2\mapsto \Omega$ where $\Omega_1$, $\Omega_2$ and $\Omega$ are respectively open neighborhoods of $0$ in $\bR^{\dim\cW_\alpha^u}$ and $\bR^{\dim\cW_\alpha^s}$ and $x$ in $M$. For any $\omega_1\in\Omega_1$, the image under $\Gamma$ of $\{\omega_1\}\times \Omega_2$ is the connected component of the intersection of $\Omega$ with $\cW_\alpha^s\big(\Gamma(\omega_1,0))$. Moreover, the pushforward of the Euclidean volume on $\Omega_1\times \Omega_2$ is absolutely continuous with respect to the Riemannian volume on $\Omega\subset M$, with H\"older continuous density function $J(\omega_1,\omega_2)$. In addition, $J$ is also $C^\infty$ along the $\cW_\alpha^s$ leaves and partial derivatives of all orders of both $\Gamma$ and $J$ along $\Omega_2$ directions are H\"older continuous. For details and references, see \cite[Theorem 2]{dlL01}.

\begin{definition}Let $\CHolder$ be the space of H\"older continuous functions $\phi$ such that $\partial_{\cW_\alpha^s}^k\phi$ are H\"older continuous functions for all $k$.
 \end{definition}

We first indicate what condition locally characterizes $\cW_\alpha^s$. For any given point $x\in M$, take a sufficiently small open neighborhood $\Omega\subset M$. We may assume that $H(\Omega)$ is a convex metric ball in the nilmanifold $M$. Then by \eqref{CorreUS}, $H^{-1}\Big(B(G^s).H(x)\Big)$ describes the connected component $\Omega^s_x$ of $x$ inside $\cW_\alpha^s(x)\cap\Omega$, where $B(G^s)$ is a small open neighborhood of the identity in $G^s$. 

As $M$ is a compact quotient of $G$, there exists $\delta>0$ and a unique $C^\infty$ function $p:\{(x,y)\in M\times M: \dist(x,y)\leq\delta\}\mapsto G$ such that $y=p(x,y)x$ and $p(x,x)=e$. 

Take $y\in B_\delta(x)$, then $H(x)=h(x)x$ and \begin{equation}\label{LocalEq}H(y)=h(y)y=h(y)p(x,y)x=H_x(y)H(x),\end{equation} with \begin{equation}\label{LocalHEq}H_x(y):=h(y)p(x,y)h(x)^{-1}.\end{equation}

\begin{remark}\label{LocalHHolder}As $H_x$ is defined on the compact subset $\{(x,y)\in M\times M: \dist(x,y)\leq\delta\}$, it follows from the smoothness of $p$ and the H\"older continuity of $h$ that $(x,y)\mapsto H_x(y)$ is H\"older continuous in the pair $(x, y)$.\end{remark}

Note that since both $p(x,\cdot)$ and $h$ are continuous and $\Omega$ is of small size, $H_x(y)$ is close to identity. So we see that if $\Omega$ is sufficiently small, then $y\in\Omega^s_x$ if and only if $H_x(y)\in G^s$.

Recall $g=g_sg_u$ for $g\in G$ and notice that $p(x,y)h(x)^{-1}$ splits as $\big(h(x)p(x,y)^{-1}\big)_u^{-1}\big(h(x)p(x,y)^{-1}\big)_s^{-1}$. Thus 
\begin{equation}\label{LocalUEq}H_x(y)=h_s(y)h_u(y)\big(h(x)p(x,y)^{-1}\big)_u^{-1}\big(h(x)p(x,y)^{-1}\big)_s^{-1}.\end{equation}
It is easy to see that $H_x(y)\in G^s$ if and only if the $G^u$ part in the middle, $h_u(y)\big(h(x)p(x,y)^{-1}\big)_u^{-1}$, is identity. This gives

\begin{lemma}\label{StableLocal}The stable leaf $\cW_\alpha^s(x)$ through a point $x$ is locally characterized by $$h_u(y)=\big(h(x)p(x,y)^{-1}\big)_u.$$\end{lemma}

\begin{corollary}\label{USmooth}$h_u\in\CHolder$.\end{corollary}
\begin{proof}We fix $x\in M$ and look at nearby points $y$.  We know $p(x,y)$ is $C^\infty$, and $h(x)$ is constant in $y$ along $\cW_\alpha^s(x)$. Furthermore, the map $a\mapsto a_u$ is smooth by Lemma \ref{SVU}. So by Lemma \ref{StableLocal}, $h_u$ is $C^\infty$ along $\cW_\alpha^s(x)$.

It remains to check that the partial derivatives along $\cW_\alpha^s$ vary H\"older continuously. To see this, fix a tiny open neighborhood $\Omega$. Actually, we can assume $\Omega$ contains a smaller open set $\Omega_0$ such that $H(\Omega_0)=B_\epsilon(G^s)B_\epsilon(G^u).H(x_0)$ for some point $x_0$, where $B_\epsilon(G^s)$, $B_\epsilon(G^u)$ are $\epsilon$-balls around identity in $G^s$ and $G^u$. This guarantees that for every $y\in\Omega_0$, one can project $y$ along $\cW_\alpha^s$ to a point $x=x(y)$ in the unstable leaf $H^{-1}\big(B_\epsilon(G^u)H(x_0)\big)\subset\cW_\alpha^u(x_0)$. The holonomy map $y\mapsto x$ is H\"older continuous as $\cW_\alpha^s$ is a H\"older foliation, and thus so is $y\mapsto h(x)$. By Lemma \ref{StableLocal}, partial derivatives of $h_u(y)$ along $\cW_\alpha^s$ depend smoothly on the pair $\big(x,h(x)\big)$,  and therefore are H\"older continuous in $y$. \end{proof}

\subsection{Solving the linearized equation}

Let $\tilde h_V=\log h_V$. Then the equation \eqref{CohoEqV} rewrites as \begin{equation}\exp\tilde h_V=\exp(\rho^{-\bfn}\tilde h_V\circ\alpha^\bfn)\exp \Psi,\end{equation}
where $\Psi$ denotes the function $\log\Big((\rho^{-\bfn}h_u\circ\alpha^\bfn)\cdot Q\Big)_V$. Both $\tilde h_V$ and $\Psi$ take values on the coarse Lyapunov subspace $\gov$. Because $Q$ is $C^\infty$ and the smooth action $\alpha^\bfn$ preserves the foliation $\cW_\alpha^s$, $\Psi\in\CHolder$ by Corollary \ref{USmooth}.

The Baker-Campbell-Hausdorff formula asserts
\begin{equation}\label{BCH}\begin{split}\tilde h_V=&\rho^{-\bfn}\tilde h_V\circ\alpha^\bfn+\Psi+\frac12[\rho^{-\bfn}\tilde h_V\circ\alpha^\bfn,\Psi]\\
&\ \ \ +\frac1{12}\Big[\rho^{-\bfn}\tilde h_V\circ\alpha^\bfn,[\rho^{-\bfn}\tilde h_V\circ\alpha^\bfn,\Psi]\Big]\\
&\ \ \ -\frac1{12}\Big[\Psi,[\rho^{-\bfn}\tilde h_V\circ\alpha^\bfn,\Psi]\Big]+\cdots\end{split}\end{equation}
where there are only finitely many terms, because we are working in a nilpotent Lie algebra and the number of brackets in each term is strictly less than the step of nilpotency.

Prior to solving this non-linear equation, we focus on its linearized form. Instead of working with $\gov$ and its expanding isomorphism $\rho^\bfn$, we will take general vector spaces.

\begin{proposition}\label{LinearSmooth}Suppose $L$ is a vector space and $\beta:L\mapsto L$ is a linear isomorphism such that $\|\beta^{-i}\|$ is uniformly bounded for all $i\geq 0$.  Let $\bfn$ be as in Lemma \ref{ExpaChoice}. If two functions $f,\psi:M\mapsto L$,  with $f$ being H\"older continuous and $\psi\in\CHolder$, satisfy \begin{equation}\label{LinearEq}f=\beta^{-1} f\circ\alpha^\bfn+\psi,\end{equation} then the solution $f$ is also in $\CHolder$. \end{proposition}

As in \cite{FKS13}*{\S 7}, we rely on the following matrix coefficient decay estimate by Gorodnik-Spatzier:

\begin{theorem}\label{MatCoef}\cite{GS} Suppose $\Sigma\subset\bZ^r$ is isomorphic to $\bZ^2$ and $\rho^\bfm$ is ergodic for every $\bfm\in\Sigma\backslash\{\bf0\}$. Then for all $\theta>0$ there are constants $\tau$ and $C$ depending on $\alpha$, $\Sigma$ and $\theta$ such that $\forall\bfm\in\Sigma\backslash\{\bf0\}$ and $\theta$-H\"{o}lder functions $f,g\in C^\theta(M)$,
$$\left|\langle f\circ\alpha^\bfm,g\rangle_{L^2(\mu)}-\int f\di \mu\int g\di\mu\right|\leq C\|f\|_{C^\theta}\|g\|_{C^\theta}e^{-\tau\theta|\bfm|}.$$\end{theorem}

Here $\mu$ is the unique $\alpha$-invariant absolutely continuous measure (see \S\ref{secFM}). And the H\"older norm $\|\cdot\|_{C^\theta}$ is given by \begin{equation}\label{HolderNorm}\|f\|_{C^\theta}=\|f\|_{L^\infty}+\sup_{x,y\in\bT^N}\frac{|f(x)-f(y)|}{\dist(x,y)^\theta}.\end{equation}

\begin{lemma}\label{distribution} In the setting of Proposition \ref{LinearSmooth}, if in addition $\int f\di\mu=0$, then $$f=\sum_{i=0}^\infty\beta^{-i}\psi\circ\alpha^{i\bfn}$$ in the sense of distributions.\end{lemma}

\begin{proof}Iterating \eqref{LinearEq} gives \begin{equation}f=\sum_{j=0}^i\beta^{-i}\psi\circ\alpha^{i\bfn}+\beta^{-i} f\circ\alpha^{i\bfn}.\end{equation}

It suffices to show that for all $\phi\in C^\infty(\bT^N)$,  $\lim_{i\rightarrow \infty}\langle\beta^{-i}f\circ\alpha^{i\bfn}, \phi\rangle_{L^2(\bT^N)}=0$. As $\mu$ has positive $C^\infty$ density function, one may replace $L^2(\bT^N)$ with $L^2(\mu)$ while proving this. 

Since $\mu$ is $\alpha$-invariant and $h$ has zero average against $\mu$,
\begin{equation}\int_{\bT^N}\beta^{-i} f\circ\alpha^{i\bfn}\di\mu=\beta^{-i}\left(\int f\di\mu\right)=0.\end{equation}
Thus we may assume $\int \phi\di\mu=0$.

Assume $f$ is $\gamma$-H\"older, then by Lemma \ref{MatCoef}, \begin{equation}\label{distributionEq1}\left|\langle\beta^{-i}f\circ\alpha^{i\bfn}, \phi\rangle_{L^2(\mu)}\right|\leq\|\beta^{-i}\|\cdot C\|f\|_{C^\gamma}\|\phi\|_{C^\gamma}e^{-\tau\gamma i|\bfn|}.\end{equation} 
As $\|\beta^{-i}\|$ is uniformly bounded by hypothesis,  \eqref{distributionEq1} decays to $0$ as $i$ grows, which completes the proof.
\end{proof}

Let $(C^\theta)^*(M)$ denote the dual space of the space $C^\theta(M)$ of the $\theta$-H\"older continuous functions on $M$.
\begin{lemma}\label{HolderDist}If $f$, $\psi$ are as in Proposition \ref{LinearSmooth} and $\int f\di\mu=0$, then for all $k\in\bN$ and all $\theta>0$, $\partial_{\cW_\alpha^s}^k f\in (C^\theta)^*$. \end{lemma}

\begin{proof} By Lemma \ref{distribution}, we only need to show \begin{equation}\label{HDeq1}\sum_{i=0}^\infty\langle \partial_{\cW_\alpha^s}^k(\beta^{-i}\psi\circ\alpha^{i\bfn}),\phi\rangle\leq C\|\phi\|_{C^\theta}.\end{equation}

Note $\psi$ is $C^\infty$ and thus $\partial_{\cW_\alpha^s}^k(\beta^{-i}\psi\circ\alpha^{i\bfn} )$ is a function, so each term in \eqref{HDeq1} is actually given by integration in terms of the Lebesgue measure. 

Fix a compactly supported positive $C^\infty$ bump function $\delta$ on $\gog$ supported on a neighborhood around $0$. For small values of $\epsilon>0$, define on $G$ a function $$\delta_\epsilon(x)=\frac{\delta(\frac{\log x}{\epsilon})}{I_\epsilon},$$ where $I_\epsilon$ is chosen such that $\int_G\delta_\epsilon(g)\di g=1$. Let $\phi_\epsilon$ be the convolution $\delta_\epsilon\star\phi$.  Like in  \cite[Eq. (11)]{FKS13}, it is easy to check the smoothification $\phi_\epsilon$  is $C^\infty$ and is well-behaved on all levels of regularity: 

\begin{equation}\label{HDeq2}\|\phi-\phi_\epsilon\|_{L^\infty}\leq a_0\epsilon^\theta\|\phi\|_{C^\theta};\end{equation}
\begin{equation}\label{HDeq3}\|\phi_\epsilon\|_{C^k}\leq c_k\epsilon^{-N-k}\|\phi\|_{L^\infty}\end{equation} for some constants $a_0$ and $c_k$ and $N=\dim M$.

By Lemma \ref{MatCoef} and \eqref{HDeq3} and the construction of distributional derivatives, for some constant $a_1>0$,
\begin{equation}\label{HDeq4}\begin{split}
\left|\langle \partial_{\cW_\alpha^s}^k(\psi\circ\alpha^{i\bfn}),\phi_\epsilon\rangle\right|
=&\left|\langle \psi\circ\alpha^{i\bfn},\partial_{\cW_\alpha^s}^k\phi_\epsilon\rangle\right|\\
\leq&a_2\left|\langle \psi\circ\alpha^{i\bfn},\partial_{\cW_\alpha^s}^k\phi_\epsilon\rangle_{L^2(\mu)}\right|\\
\leq&a_1a_2\|\psi\|_{C^\theta}\|\partial_{\cW_\alpha^s}^k\phi_\epsilon\|_{C^\theta}e^{-\tau\theta i|\bfn|}\\
\leq&a_1a_2\|\psi\|_{C^\theta}\|\phi_\epsilon\|_{C^{k+1}}e^{-\tau\theta i|\bfn|}\\
\leq&a_1a_2c_{k+1}\|\psi\|_{C^\theta}\|\phi\|_{L^\infty}\epsilon^{-N-k-1}e^{-\tau\theta i|\bfn|}\\
\leq&C_1\epsilon^{-N-k-1}e^{-\tau\theta i|\bfn|}\|\phi\|_{C^\theta},\end{split}\end{equation}
where $a_2$ is the supremum of the density function of the absolutely continuous measure $\mu$. The constant $C_1=a_1a_2c_k\|\psi\|_{C^\theta}$ depends on $k$, $\theta$, $\bfn$ and $\psi$, but is independent of $\phi$.

On the other hand,  by \eqref{HDeq2}
\begin{equation}\label{HDeq5}\begin{split}
&\left|\langle \partial_{\cW_\alpha^s}^k(\psi\circ\alpha^{i\bfn}),\phi-\phi_\epsilon\rangle\right|\\
\leq &\left\|\partial_{\cW_\alpha^s}^k(\psi\circ\alpha^{i\bfn})\right\|_{L^\infty}\|\phi-\phi_\epsilon\|_{L^\infty}\\
\leq&a_0\left\|\partial_{\cW_\alpha^s}^k\psi\right\|_{L^\infty}\left\|\partial_{\cW_\alpha^s}^k\alpha^{i\bfn}\right\|_{L^\infty}\epsilon^\theta\|\phi\|_{C^\theta}.
\end{split}\end{equation}

By \cite[Lemma 3.6]{FKS13}, \begin{equation}\label{HDeq6}\begin{split}
\left\|\partial_{\cW_\alpha^s}^k\alpha^{i\bfn}\right\|_{L^\infty}\leq& a_3\|D\alpha^\bfn|_{E_\alpha}\|^{ik}i^T\left\|\partial_{\cW_\alpha^s}^k\alpha^\bfn\right\|_{L^\infty}^T\\
\leq& a_3e^{\eta ik|\bfn|}i^T\left\|\partial_{\cW_\alpha^s}^k\alpha^\bfn\right\|_{L^\infty}^T\\
\leq& a_3a_4e^{2\eta ik|\bfn|}\left\|\partial_{\cW_\alpha^s}^k\alpha^\bfn\right\|_{L^\infty}^T,\end{split}\end{equation} 
where constants $a_3$ and $T$ depend only on $k$ and $\dim\cW_\alpha^s$, and $a_4=a_4(\eta,T)$. The second inequality is based on the choice of $\bfn$ from Lemma \ref{ExpaChoice}.
Combining \eqref{HDeq5}, \eqref{HDeq6}, we get
 \begin{equation}\label{HDeq7}
\left|\langle \partial_{\cW_\alpha^s}^k(\psi\circ\alpha^{i\bfn}),\phi-\phi_\epsilon\rangle\right|\leq a_4 C_2\epsilon^\theta e^{2\eta ik|\bfn|}\|\phi\|_{C^\theta},\\
\end{equation} where the constant $C_2=a_0a_3\left\|\partial_{\cW_\alpha^s}^k\psi\right\|_{L^\infty}\left\|\partial_{\cW_\alpha^s}^k\alpha^\bfn\right\|_{L^\infty}^T$ depends only on $k$, $\bfn$, $\alpha$ and $\psi$, and $a_4$ depends on $k$, $\alpha$ and $\eta$. 

Pick $\epsilon=e^{\frac{-\big(\tau+2\eta k\big)\theta i|\bfn|}{N+k+1-\theta}}$, then $\epsilon^{-N-k-1}e^{-\tau\theta i|\bfn|}$ and $\epsilon^\theta e^{-2\eta ik|\bfn|}$ are both equal to $e^{\frac{\big(-\tau\theta+2\eta(N+k+1)k\big)\theta i|\bfn|}{N+k+1+\theta}}$. One may fix a sufficiently small $\eta$, again in a way that depends only on $k$, $\bfn$, the subgroup $\Sigma$ in Lemma \ref{ExpaChoice} and the action $\alpha$, so that $\tau\theta-2\eta(N+k+1)k<0$.  Then $a_4$ is also determined by $k$, $\bfn$ and $\alpha$.

Denote $a_5=\frac{\big(-\tau\theta+2\eta(N+k+1)k\big)\theta |\bfn|}{N+k+1+\theta}$, then $a_5>0$ and is independent of $i$ and $\phi$.  Hence merging \eqref{HDeq4} and \eqref{HDeq7} gives \begin{equation}\label{HDeq8}\left|\langle (\psi\circ\alpha^{i\bfn})^{k,\cW},\phi\rangle\right|\leq (C_1+a_4C_2)e^{-a_5i}\|\phi\|_{C^\theta}, \forall i\in\bN.\end{equation}

By assumption, $\|\beta^{-i}\|$ is uniformly bounded by some constant $a_6$.  Therefore \eqref{HDeq8} implies the right hand side of \eqref{HDeq1} is bounded by $$\sum_{i=0}^\infty a_6(C_1+a_4C_2)e^{-a_5i}\|\phi\|_{C^\theta}=\frac{a_6(C_1+a_4C_2)}{1-e^{-a_5}}\|\phi\|_{C^\theta}$$ and this completes the proof.\end{proof}

\begin{proof}[Proof of Proposition \ref{LinearSmooth}] When $\int f\di\mu=0$, Lemma \ref{HolderDist} applies and we know partial derivatives of all orders of $f$ along $\cW_\alpha^s$ are in $(C^\theta)^*$ for all $\theta>0$. In this case, all these derivatives are actually H\"older continuous functions and hence $f\in\CHolder$. This is an application of Theorem \ref{RauchTaylor}, which is an extension of a theorem of Rauch and Taylor \cite{RT05}.

The general case reduces easily to the situation above. Actually, set $\bar f=\int f\di\mu$ and $f_1=f-\bar f$. As $\bar f\in L$ is a constant, it suffices to prove $f_1\in\CHolder$. $f_1$ has zero average against $\mu$. Moreover, it satisfies the equation $$f_1=\beta^{-1}f_1\circ\alpha+(\psi-\bar f+\beta^{-1}\bar f\circ\alpha^\bfn).$$ 

$\psi-\bar f+\beta^{-1}\bar f\circ\alpha^\bfn$ differs from $\psi$ by a constant and is thus in $\CHolder$. Given the zero average case, $f_1$ is in $\CHolder$.
\end{proof} 

\subsection{Establishing smoothness} We now complete the main task of this section by proving:

\begin{proposition}\label{hVSmooth}$h_V\in\CHolder$.\end{proposition} 

This is equivalent to that $\tilde h_V=\log h_V$ is in $\CHolder$. The later claim will be proved by decomposing $\tilde h_V$ into components of different nilpotent steps.

As a Lie subalgebra of $\gog$, $\gov$ is nilpotent. Take the lower central series $$\gov=\gov_1\supset\gov_2\supset\cdots\supset\gov_{l+1}=\{0\}$$ where $[\gov_i,\gov_j]\subset\gov_{i+j}$ for all $i,j$ and each $\gov_i$ is an ideal of $\gov$. The projection $\pi_i$ from $\gov$ into the quotient Lie algebra $\gov/\gov_i$ is a Lie algebra morphism, i.e., $[\pi_iA,\pi_iB]=\pi_i[A,B]$. Write $\tilde h_i=\pi_i\circ\tilde h$, $\psi_i=\pi_i\circ\psi$.

Remark that since $\rho$ acts by automorphisms, all the $\gov_i$'s are $\rho$-invariant. The induced action on $\gov/\gov_i$ will again be called $\rho$.

\begin{lemma}For $i=1,\cdots,l+1$, $\tilde h_i\in\CHolder$.\end{lemma}
\begin{proof} The proof is by induction on $i$. When $i=1$, the claim is obvious as $\tilde h_0$ maps to the trivial space $\gov/\gov$.

Assume $i\geq 2$ and the lemma holds for all $j\leq i$. Projecting equation \eqref{BCH} into $\gov/\gov_i$, we get
\begin{equation}\label{BCHi}\tilde h_i=\rho^{-\bfn}\tilde h_i\circ\alpha^\bfn+\Psi_i+\frac12[\rho^{-\bfn}\tilde h_i\circ\alpha^\bfn,\Psi_i]+\cdots\end{equation}

Fix a subspace $\goz\subset\gov/\gov_i$, such that $\goz\oplus(\gov_{i-1}/\gov_i)=\gov$. Split $\tilde h_i$ as $\tilde h_\goz+\tilde h_\goz^\bot$ accordingly.

Observe that $\goz$ is not a Lie subalgebra, but the natural projection from $\goz$ into $\gov/\gov_{i-1}$ is a linear isomorphism. Therefore, by the inductive hypothesis on $\tilde h_{i-1}$, $\tilde h_\goz$ is in $\CHolder$. Since $\rho^\bfn$ is linear and $\alpha^\bfn$ preserves smoothly the foliation $\cW_\alpha^s$, we know $\rho^{-\bfn}\tilde h_\goz\circ\alpha^\bfn\in\CHolder$.  Recall that $\Psi$ is also in $\CHolder$.

Each higher order term in \eqref{BCHi} can be decomposed into a finite sum of repeated Lie bracket monomials of the form $[\square,[\cdots,[\square,\square]]\cdots]$, where each $\square$ is one of $\rho^{-\bfn}\tilde h_\goz\circ\alpha^\bfn$, $\rho^{-\bfn}\tilde h_\goz^\bot\circ\alpha^\bfn$, and $\Psi_i$. Because $\rho^{-\bfn}\tilde h_\goz^\bot\circ\alpha^\bfn$ is from $\gov_{i-1}/\gov_i$, all Lie brackets involving it are in $\gov_i$ and thus vanish modulo $\gov_i$. Therefore, the bracket terms form a polynomial in $\rho^{-\bfn}\tilde h_\goz\circ\alpha^\bfn$ and $\Psi_i$, which is a $\CHolder$ function as both $\rho^{-\bfn}\tilde h_\goz\circ\alpha^\bfn$ and $\Psi_i=\pi_i\circ\Psi$ are.

For this reason, \eqref{BCHi} can be written as a linear equation \begin{equation}\tilde h_i=\rho^{-\bfn}\tilde h_i\circ\alpha^\bfn+\tilde\Psi_i,\end{equation} where $\tilde\Psi_i=\Psi_i+\text{[bracket terms]}$ belongs to $\CHolder$. Because of the choice of $\bfn$ in Lemma \ref{ExpaChoice}, the Lyapunov exponent of $\rho^\bfn$ in $\gov$ is positive and it follows that $\|\rho^{-i\bfn}|_\gov\|$ is uniformly bounded for $i\geq 0$. This boundedness is inherited when $\rho^{-i\bfn}$ is projected to $\gov/\gov_i$. Hence Proposition \ref{LinearSmooth} applies to this case, implying $\tilde h_i\in\CHolder$.\end{proof}

\begin{proof}[Proof of Proposition \ref{hVSmooth}] $\tilde h_V\in\CHolder$ by setting $i=l+1$ in the previous lemma. And $h_V=\exp\tilde h_V$ is in the same class, as $\exp$ is smooth.\end{proof}

\section{Construction of new Anosov elements}\label{secAno}

In this section, we explore the geometric meaning of Proposition \ref{hVSmooth}. The smoothness of $h_V$ along $\cW_\alpha^s$ will guarantee the existence of a continuous subfoliation inside $\cW_\alpha^s$ with smooth leaves, which consists of the preimages under the conjugacy $H$ of the orbits of the Lie subgroup $G^{ss}$. And  in most cases\footnote{In the non-generic case where there is a non-trivial coarse Lyapunov subspace $\gov^{[-\chi]}$ whose Lyapunov exponents are negatively proportional to that of $\gov$, the new stable foliation corresponds to $G^{ss}\cdot V^{[-\chi]}$, also a Lie subgroup, instead. }, this will be the stable foliation of a new Anosov element which lies in the Weyl chamber $\cC$ that neighbors $\cC_0$ along $\ker\chi$.

\subsection{Local description of the strong stable foliation} We want to describe locally the topological submanifold
\begin{equation}\label{ssFoliation}\cW_\alpha^{ss}(x):=H^{-1}(G^{ss}.H(x)).\end{equation} Like in \S\ref{sechuSmooth}, we fix a sufficiently small neighborhood
 $\Omega$ around $x$. By repeating the reasoning in \S\ref{sechuSmooth}, especially \eqref{LocalEq}, we see that a point $y\in\Omega$ belongs to the connected component $\Omega^{ss}_x$ of $\cW_\alpha^{ss}(x)\cap\Omega$, if and only if $H_x(y)\in G^{ss}$, where $H_x$ was defined in \eqref{LocalHEq}. We emphasize that $H_x$ is only locally defined near $x$.
 
As $G^{ss}\subset G^s$, $\cW_\alpha^{ss}(x)\subset\cW_\alpha^s(x)$. For this reason, we restrict to points $y$ from $\Omega^s_x$, the connected neighborhood of $\cW_\alpha^s\cap\Omega$. 

Under this assumption, $H_x(y)\in G^s$. Thus $G^{ss}$ contains $H_x(y)=\big(H_x(y)\big)_{ss}\big(H_x(y)\big)_V$ if and only if $\big(H_x(y)\big)_V$ is identity. Therefore, we obtain the following lemma.

\begin{lemma}\label{SSLocal} Inside a stable leaf $\cW_\alpha^s(x)$, the strong stable leaf $\cW_\alpha^{ss}(x)$ is locally characterized near $x$ by the equation $$\big(H_x(y)\big)_V=e.$$\end{lemma}

When $y$ is in the connected neighborhood $\Omega^s_x\subset\cW_\alpha^s(x)$ around $x$, we know from \S\ref{sechuSmooth} that $h_u(y)\big(h(x)p(x,y)^{-1}\big)_u^{-1}=e$, and it follows from \eqref{LocalUEq} that,
\begin{equation}\label{LocalSEq}H_x(y)=h_s(y)\big(h(x)p(x,y)^{-1}\big)_s^{-1}.\end{equation}

\begin{proposition}\label{HVSmooth}\begin{enumerate}\item The map $y\mapsto \big(H_x(y)\big)_V$ is $C^\infty$ in sufficiently small neighborhoods of $x$ in $\cW_\alpha^s(x)$;
\item The partial derivatives $\partial_{\cW_\alpha^s}^k|_{y=x}\big(H_x(y)\big)_V$ are H\"older continuous for all $k$.\end{enumerate}\end{proposition}
\begin{proof}(1) By Corollary \ref{Vpart} (\ref{Vpart2}) and \eqref{LocalSEq}, \begin{equation}\label{LocalVEq}\big(H_x(y)\big)_V=h_V(y)\Big(\big(h(x)p(x,y)^{-1}\big)_s^{-1}\Big)_V.\end{equation} Notice that the function $h(x)p(x,y)^{-1}$ is $C^\infty$ in $y$, hence by Lemma \ref{SVU}, $y\mapsto  \Big(\big(h(x)p(x,y)^{-1}\big)_s^{-1}\Big)_V$ is smooth.  Using \eqref{LocalEq}, Proposition \ref{hVSmooth} immediately implies (1).

(2) By \eqref{LocalVEq}, partial derivatives $\partial_{\cW_\alpha^s}^k\big(H_x(y)\big)_V$ are polynomial combinations of:
\begin{itemize}\item Partial derivatives $\partial_{\cW_\alpha^s}^kh_V(y)$, which are H\"older continuous in $y$ by Proposition \ref{hVSmooth} and independent of $x$; and 
\item Partial derivatives of the $C^\infty$ function $y\mapsto\Big(\big(h(x)p(x,y)^{-1}\big)_s^{-1}\Big)_V$ along $\cW_\alpha^s$, which are smooth in $y$ and depends H\"older continuously on $x$ (following the same argument from the proof of Corollary \ref{USmooth}). \end{itemize}
Therefore, when one sets $y=x$, these derivatives of $\big(H_x(y)\big)_V$ have H\"older continuous dependence on $x$.\end{proof}

Combining Lemma \ref{SSLocal} and Proposition \ref{HVSmooth} yields the following criterion:

\begin{corollary}\label{RegSmooth}$\cW_\alpha^{ss}$ has $C^\infty$ leaf at $x$ if $\big(H_x(y)\big)_V$ is regular in $y$ at $y=x$, that is, the map $D_{\cW_\alpha^s}|_{y=x}\big(H_x(y)\big)_V:E_\alpha^s(x)\mapsto\gov$ has rank equal to $\dim\gov$.\end{corollary}

Let $A$ be the set of points $x\in M$ at which $\big(H_x(y)\big)_V$ is singular at $x$, i.e. not regular at $x$.

\begin{lemma}$A$ is closed and invariant under the $\bZ^r$-action $\alpha$.\end{lemma}
\begin{proof} We know that $D_{\cW_\alpha^s}|_{y=x}\big(H_x(y)\big)_V$ depends continuously on $x$. Furthermore, being singular, or equivalently the linear map $D_{\cW_\alpha^s}|_{y=x}\big(H_x(y)\big)_V$ being degenerate, is a closed condition. It follows that $A$ is closed. We now check the $\alpha$-invariance.

Fix $\bfm\in\bZ^r$. Consider two points $x$ and $\alpha^\bfm x$, as well as $y$ and $\alpha^\bfm y$, respectively in sufficiently small neighborhoods of $x$ and $\alpha^\bfm x$ in their $\cW_\alpha^s$-leaves. Using the conjugacy relation $\rho\circ H=H\circ\alpha$, we relate $H_x(y)$ to $H_{\alpha^\bfm x}(\alpha^\bfm y)$. Recall the of definition $H_x(y)$ in (\ref{LocalEq}), $H(y)=H_x(y)H(x)$
\begin{equation}\begin{split}
H_{\alpha^\bfm x}(\alpha^\bfm y)H(\alpha^\bfm x)=&H(\alpha^\bfm y)=\rho^\bfm H(y)\\
=&\rho^\bfm \big(H_x(y)H(x)\big)\\
=&\big(\rho^\bfm H_x(y)\big)\big( \rho^\bfm H(x)\big)\\
=&\big(\rho^\bfm H_x(y)\big)H(\alpha^\bfm x).
\end{split}\end{equation}
Since $\bfm$ is fixed and both $H_x(y)$ and $H_{\alpha^\bfm x}(\alpha^\bfm y)$ are sufficiently close to identity, we see the group elements $H_{\alpha^\bfm x}(\alpha^\bfm y)$, $\rho^\bfm H_x(y)\in G$ are equal, and in particular, as the $G^{ss}VG^u$-decomposition is $\rho$-invariant, \begin{equation}\label{HVConjEq}\big(H_{\alpha^\bfm x}(\alpha^\bfm y)\big)_V=\rho^\bfm \big(H_x(y)\big)_V.\end{equation} 

Since $\cW_\alpha^s$ is $\alpha$-invariant, it follows that \begin{equation}\label{DerAct}\begin{split}&D_{\cW_\alpha^s}|_{y=\alpha^\bfm x}\big(H_{\alpha^\bfm x}(y)\big)_V=\\
&\ \ \ \ \ \ \ \rho^\bfm|_\gov\left(D_{\cW_\alpha^s}|_{y=x}\big(H_x(y)\big)_V\right)\big(D_{\alpha^\bfm x}\alpha^{-\bfm}\big)|_{E_\alpha^s}.\end{split}\end{equation}
Where $\rho^\bfm|_\gov$ is the restriction of the induced action by $\rho$ on $\gog$ to the invariant Lie subalgebra $\gov$. In consequence, if $D_{\cW_\alpha^s}|_{y=x}\big(H_x(y)\big)_V$ is not of full rank, then neither is $D_{\cW_\alpha^s}|_{y=\alpha^\bfm x}\big(H_{\alpha^\bfm x}(y)\big)_V$. This means $A$ is $\alpha$-invariant.\end{proof}

\subsection{Coarse Lyapunov decomposition over an invariant measure}

We hope to prove the set $A$ is empty.

\begin{proposition}\label{ssSmooth}For every $x\in X$, $\big(H_x(y)\big)_V$ is regular at $x$. In consequence, $\cW_\alpha^{ss}$ is a continuous foliation of $M$ with $C^\infty$ leaves.\end{proposition}

Assume, for contradiction, that the proposition fails, then $A$ is a non-empty $\alpha$-invariant closed subset of $M$, and hence supports an ergodic $\alpha$-invariant measure $\nu$ (observe that $\nu$ may be an atomic measure, i.e. supported on a finite orbit).

In this case, Oseledets' multiplicative ergodic theorem can be adapted to the $\bZ^r$-action $\alpha$ (see \cite{KSad06}*{Prop. 2.1} and \cite{KS07}). In consequence, there are finitely many linear functionals $\xi\in(\bR^r)^*$, an $\alpha$-invariant subset $A'\subset A$ with $\nu(A')=1$ and an $\alpha$-invariant measurable splitting \begin{equation}\label{OseledetsDecomp}T_xM=\bigoplus E_\nu^\xi(x)\end{equation} over $A'$ such that, for all $\bfn\in\bZ^r$ and $v\in E_\nu^\xi$, 
\begin{equation}\label{OseledetsEq}\lim_{k\rightarrow\pm\infty}\frac{\log|(D\alpha^{k\bfn})v|}k=\xi(\bfn).\end{equation}

Moreover, by modifying $A'$ if necessary, we have:

\begin{lemma}\label{PesinSS}In the Oseledets decomposition above, it can be required that for all $x\in A'$ and $\bfn\in\bZ^r$, there are unique submanifolds $\cW_{\alpha^\bfn,\nu}^s(x)$ and $\cW_{\alpha^\bfn,\nu}^u(x)$ respectively tangent to the stable and unstable distributions $$E_{\alpha^\bfn,\nu}^s:=\bigoplus_{\xi(\bfn)<0}E_\nu^\xi, E_{\alpha^\bfn,\nu}^u:=\bigoplus_{\xi(\bfn)>0}E_\nu^\xi, $$ In addition, a neighborhood of $x$ in $\cW_{\alpha^\bfn,\nu}^s(x)$ is given by the set of points $y$ satisfying $\dist(x,y)<\epsilon$ for some $\epsilon=\epsilon(x)$, and $$\limsup_{k\rightarrow\infty}\log\dist(\alpha^{k\bfn}x,\alpha^{k\bfn}y)\leq\max_{\{\xi:\xi(\bfn)<0\}}\xi(\bfn).$$\end{lemma}

\begin{proof}This is Pesin's strong stable manifold theorem. See for instance \cite{RH07}*{Theorem 3.2} and \cite{R79}.\end{proof}

\begin{lemma}\label{DecompCompare}For any fixed $\bfn\in\bZ^r$, if $\nu$ is a hyperbolic measure for $\bfn$, i.e. $\xi(\bfn)\neq 0$ for all Lyapunov functionals $\xi$ in \eqref{OseledetsDecomp}, then for $\square=s,u$, 
\begin{enumerate}\item $\dim E_{\alpha^\bfn,\nu}^\square=\dim E_{\rho^\bfn}^\square=\dim\gog_{\rho^\bfn}^\square$;
\item $H\big(\cW_{\alpha^\bfn,\nu}^\square(x)\big)=\cW_{\rho^\bfn}^\square\big(H(x)\big)=G_{\rho^\bfn}^\square.H(x)$.
\end{enumerate}\end{lemma} 
\begin{proof}Knowing that $H$ is a bi-H\"older conjugacy between the actions $\rho$ and $\alpha$, the statements follow directly from  \cite{RH07}*{Prop. 3.1 \& Cor. 3.3}.\end{proof}

Coarse Lyapunov distributions are defined in a similar way to Definition \ref{CoarseDef}: 
\begin{definition}\label{CoarseDef2}$\displaystyle E_\nu^{[\xi]}=\bigoplus_{\xi'=c\xi,c>0}E_\nu^{\xi'}.$\end{definition}

\begin{lemma}\label{CoarseProportional} The coarse Lyapunov subspaces in Definitions \ref{CoarseDef} and \ref{CoarseDef2} are in one-to-one correspondence to each other. A pair of corresponding coarse Lyapunov subspaces have the same dimension and proportional coarse Lyapunov exponents. \end{lemma}
\begin{proof}Because $\alpha^{\bfn_0}$ is Anosov, $\xi(\bfn_0)\neq 0$ for all Lyapunov functionals $\xi$ with respect to $\nu$. In particular, the $\xi$'s are all non-trivial. The same is true for all Lyapunov functionals $\chi$ in the linear decomposition \ref{LyaDecompEq}.

One can find a finite set $\Delta$ of $\bfn\in\bZ^r$ which represents all Weyl chambers with respect to the coarse Lyapunov decompositions in both Definitions \ref{CoarseDef} and \ref{CoarseDef2}.

Then the coarse Lyapunov subspaces $E_\nu^{[\xi]}$ and $\gov^{[\chi]}$ are respectively minimal non-trivial intersections between distinct stable subspaces $E_{\alpha^\bfn,\nu}^s$ and $E_{\rho^\bfn}^s$ as $\bfn$ varies within $\Delta$. Note as $\bfn\in\Delta$ are chosen from the interior of Weyl chambers, the measure $\nu$ is hyperbolic with respect to $\alpha^\bfn$. It follows from Lemma \ref{DecompCompare} that $E_\nu^{[\xi]}$ and $\gov^{[\chi]}$ are bijectively associated to each other. 

In particular, $E_\nu^{[\xi]}$ is stable (resp. unstable) with respect to $\alpha^\bfn$ if and only if $\gov^{[\chi]}$ is stable (reps. unstable) with respect to $\rho^\bfn$. In other words, $\xi(\bfn)>0$ if and only if $\chi(\bfn)>0$ for all $\bfn\in\bZ^r$, which forces the functionals $\xi$ and $\chi$ to be positively proportional to each other. The dimensions are equal thanks to Lemma \ref{DecompCompare}.\end{proof}

Hence whenever Lyapunov functionals $\chi'$ and $\xi'$ are respectively from a pair of corresponding coarse Lyapunov subspaces $E_\nu^{[\xi]}$ and $\gov^{[\chi]}$, $\chi'=c\xi'$ for some $c>0$. We can choose $\kappa>1$ such that $\kappa^{-1}<c<\kappa$ for all such pairs. 

Since $\alpha^{\bfn_0}$ is Anosov we get that for $\square=s,u$, $E_{\alpha^{\bfn_0},\nu}^\square=E_\alpha^\square$ which we shall denote also with $E^\square_\nu$.

%Following similar notations that we have been using, denote $E_\nu^\square=E_{\alpha^{\bfn_0},\nu}^\square$ for $\square=s,u$. 

Given the decomposition $\gog^s=\gog^{ss}\oplus\gov$, we now have a corresponding decomposition \begin{equation}E_\nu^s=E_\nu^{ss}\oplus E_\nu^V\end{equation} at $\nu$-almost every $x$. Here, $E_\nu^V=E_\nu^{[\xi]}$ is the coarse Lyapunov subspace associated to $\gov=\gov^{[\chi]}$ by Lemma \ref{CoarseProportional}. The subspace $E_\nu^{ss}$ is the direct sum of all the coarse Lyapunov subspaces $E_\nu^{[\xi']}$ where $\xi'(\bfn_0)<0$ but $\xi'$ is not proportional to $\xi$.

It follows from the proof of Lemma \ref{SUsubalg} that there is $\bfz\in\bR^r$ that satisfies $\chi(z)=0$ and \eqref{SUsubalgEq}. Therefore, there exists $\lambda>0$ such that for all $\eta>0$, one may slightly perturb $\bfz$ to get $\bfm\in\bR^r$ such that
\begin{equation}\label{ContraChoiceEq}\begin{split}
\chi'(\bfm)\in(-\eta|\bfm|,0)&\text{ if }\gov^{\chi'}\subset\gov;\\
\chi'(\bfm)<-\lambda|\bfm|&\text{ if }\gov^{\chi'}\subset\gog^{ss};\\
\chi'(\bfm)>0&\text{ if }\gov^{\chi'}\subset\gog^u.
\end{split}\end{equation}
As $\bQ^r$ is dense in $\bR^r$, without loss of generality, one may assume $\bfm$ belongs to $\bZ^r$. $\bfm$ is an element in the Weyl chamber $\cC_0$, that makes a very small angle with $\ker\chi$.

As remarked in \S\ref{secSetting}, $\cW_\alpha^s$ and $\cW_\alpha^u$ are invariant under $\alpha^\bfm$. Using \eqref{CorreUS} and \eqref{ContraChoiceEq}, Lemma \ref{Subadd2} implies that $\cW_\alpha^s$ and $\cW_\alpha^u$ are respectively the stable and unstable foliations of $\alpha^\bfm$. So $\alpha^\bfm$ is Anosov.

By Lemma \ref{CoarseProportional}, the Lyapunov exponents of $\alpha^\bfm$ with respect to $E_\nu^u$, $E_\nu^V$ and $E_\nu^{ss}$  are respectively in the intervals $(0,\infty)$, $(-\kappa\eta|\bfm|,0)$, $(-\infty,-\kappa^{-1}\lambda|\bfm|)$. Given the values of $\lambda$ and $\kappa$, by picking a sufficiently small $\eta$, these intervals can be made disjoint from each other and $\frac{-\kappa^{-1}\lambda|\bfm|}{-\kappa\eta|\bfm|}$ can be made arbitrarily large.

For $\square\in\{ss,V,u\}$, set $N^\square=\dim E_\nu^\square$, which are equal to $\dim G^{ss}$, $\dim V$ and $\dim G^{ss}$ respectively and let $\lambda_-^\square$ and $\lambda_+^\square$ denote respectively the smallest and largest Lyapunov exponent of $\alpha^\bfm$ in the $E_\nu^\square$ direction. Furthermore, let $\chi_-^V$ be the smallest Lyapunov exponents of $\rho^\bfm$ with respect to $\gov$. Then for all $\kappa'>1$, by \eqref{ContraChoiceEq}, $\bfm$ can be chosen in such a way that 
\begin{equation}\label{Bands}\lambda_-^{ss}\leq\lambda_+^{ss}<\lambda_-^V\leq\lambda_+^V<0<\lambda_-^u\leq\lambda_+^u,\end{equation}
and \begin{equation}\label{Bands2}\lambda_+^{ss}<\kappa'\chi_-^V<0.\end{equation}

\subsection{Smoothness of strong stable leaves} In this part, we complete the proof of Proposition \ref{ssSmooth}.

Thanks to Pesin theory, at $\nu$-almost every $x$, the stable manifold $\cW_\alpha^s(x)$ can be locally foliated in a way that corresponds to the foliation of $W_\rho^s$ by $G^{ss}$-orbits. We now  describe this foliation, following the formulation of F. Ledrappier and L. S. Young. 

Let $\nu'$ be an arbitrary $\alpha^\bfm$-ergodic component of $\alpha^\bfm$. Set \begin{equation}\label{Bands3}\epsilon_0=\frac{\min\{\frac{\chi_-^V}{\kappa'}-\lambda_+^{ss}, \lambda_-^V-\lambda_+^{ss},-\lambda_+^V,\lambda_-^u\}}{1000}\end{equation}

\begin{lemma}\label{LyaChart}{\rm (Lyapunov charts)} Let $A'$ be as in Lemma \ref{PesinSS} and $\bfm$ be as in \eqref{ContraChoiceEq}. For all $0<\epsilon<\epsilon_0$, one can choose an $\alpha^\bfm$-invariant subset $A''\subset A'$ with $\nu'(A'')=1$ such that the following conditions are satisfied.  For all $0<\epsilon<\epsilon_0$, there exist  a constant $c>0$, a measurable function $l:A'\mapsto(1,\infty)$ and an embedding $\Phi_x: B_{\bR^N}(l(x)^{-1})\mapsto M$ such that:
\begin{enumerate}
\item $\log\dfrac{l(\alpha^\bfm x)}{l(x)}\in(-\epsilon,\epsilon)$;
\item $\Phi_x0=x$, and $D_0\Phi_x$ sends the splitting $\bR^{N^{ss}}\oplus\bR^{N^V}\oplus\bR^{N^u}$ to $E_\nu^{ss}\oplus E_\nu^V\oplus E_\nu^u$;
\item Set $f_x=\Phi_{\alpha^\bfm x}^{-1}\circ\alpha^\bfm\circ\Phi_x$ and $f_x^{-1}=\Phi_{\alpha^{-\bfm} x}^{-1}\circ\alpha^{-\bfm}\circ\Phi_x$ whenever the expression makes sense, then for $\square=u,V,ss$ and all non-zero vector $v\in\bR^{N^\square}$,
$$\log\frac{|(D_0f_x)v|}{|v|}\in(\lambda_-^\square-\epsilon,\lambda_+^\square+\epsilon) ;$$
\item Both $f_x-D_0f_x$ and $f_x^{-1}-D_0f_x^{-1}$ are $\epsilon$-Lipschitz ;
\item For $z,z\in B_{\bR^N}\big(l(x)^{-1}\big)$, $c<\dfrac{|z-z'|}{\dist(\Phi_xz,\Phi_xz')}<l(x)$.
\item All $x\in A''$ are $\nu'$-generic under $\alpha^\bfm$ in terms of Birkhoff ergodic average: $\lim_{T\rightarrow\infty}\frac1T\sum_{k=0}^{T-1}\delta_{\alpha^{k\bfm} x}=\nu'$ in weak$^*$ topology.
\end{enumerate} \end{lemma}

These charts are given on a $\nu'$-full measure set $A''$ in \cite{LY85}*{\S 8.1}. As $\nu'(A'')=1$, it is always possible to modify $A''$ so that Birkhoff Ergodic Theorem holds for $\alpha^\bfm$ at every $x\in A''$.

\begin{lemma}\label{SSLeaf}\cite{LY85}*{Lemma 8.2.3 \& 8.2.5} There exists $\tau\in(0,\frac12)$ such that, for all $x\in A''$, for all $y\in\cW_\alpha^s(x)\cap \Phi_x\big(B_{\bR^N}(\tau l(x)^{-1})\big)$ there is a map $g_{x,y}:B_{\bR^{N^V}}\big(l(x)^{-1}\big)\mapsto\bR^{N^{ss}}$ such that:
\begin{enumerate}
\item $\Phi_x^{-1}y\in \graph(g_{x,y})$;
\item $\|Dg_{x,y}\|\leq\frac13$;
\item Let $f_x^k=f_{\alpha^{(k-1)\bfm} x}\circ\cdots f_x$ where $f_x$ is defined as in Lemma \ref{LyaChart}. Then for all $z\in \mathrm{graph}(g_{x,y})$, $$\overline\lim_{k\rightarrow\infty}\frac1k\log\big|f_x^kz-f_x^k\Phi_x^{-1}y\big|\leq \lambda_+^{ss}+\epsilon.$$
\end{enumerate}
\end{lemma}

The lemma describes the strong stable subfoliation inside the stable manifolds. It should be emphasized that this is very different from the strong stable foliation that one gets from Pesin's strong stable manifold theorem, which is defined on a set of generic points. The foliation of $\bR^{N^s}$, which is identified with the $\cW_\alpha^s$-leaf at a generic point $x$, by graphs of $g_{x,y}$, defines strong stable manifolds at every nearby point $y$ inside $\cW_\alpha^s(x)$, while $y$ itself may not be generic in the Oseledets' sense or in the support of $\nu$. For this reason the next Lemma is not a straightforward consequence interplay of the conjugacy and coarse Lyapunov directions.

\begin{lemma}\label{HVConst}The function $(H_x)_V\circ\Phi_x$ is constant along the graph of $g_{x,y}$ for all $y\in B_{\bR^{N^s}}\big(l(x)^{-1}\big)$.\end{lemma}

\begin{proof} Endow the group $V$ with the Euclidean metric of $\gov$, which we identify with $V$ by exponential map. We deduce from \eqref{HVConjEq} that
\begin{equation}\label{HVConstEq}\begin{split}
&\dist\Big(\rho^{k\bfm}\big(H_x(\Phi_xz)\big)_V,\rho^{k\bfm}\big(H_x(y)\big)_V\Big)\\
=&\dist\Big(\big(H_{\alpha^{k\bfm}x}(\alpha^{k\bfm}\Phi_x z)\big)_V,\big(H_{\alpha^{k\bfm}x}(\alpha^{k\bfm}y\big)_V\Big)\\
=&\dist\Big(\big(H_{\alpha^{k\bfm}x}(\Phi_{\alpha^{k\bfm}x}f_x^k z)\big)_V, \big(H_{\alpha^{k\bfm}x}(\Phi_{\alpha^{k\bfm}x}f_x^k\Phi_x^{-1}y)\big)_V\Big).\end{split}\end{equation} 

Choose $l_0$ such that $\nu'(\{x\in A'': l(x)<l_0\})>0$. Thanks to part (6) of Lemma \ref{LyaChart}, there is an infinite sequence $\cS$ of positive integers $k$ such that $l(\alpha^{k\bfm}x)<l_0$. 

When $k$ is from $\cS$ and tends to $\infty$, $\Phi_{\alpha^{k\bfm} x}$ is uniformly Lipschitz. The distance between $\Phi_{\alpha^{k\bfm}x}f_x^k z$ and $
\Phi_{\alpha^{k\bfm}x}f_x^k\Phi_x^{-1}y$ decays exponentially at rate $e^{\lambda_+^{ss}+\epsilon}$ or faster by Lemma \ref{SSLeaf}. 

By Remark \ref{LocalHHolder}, $(H_x)_V$ is H\"older continuous, where the H\"older exponent and coefficient are uniform in $x$. We can choose $\bfm$ from the beginning such that the coefficient $\kappa'$ in \eqref{Bands2} and \eqref{Bands3} is greater than the H\"older exponent of $(H_x)_V$ for all $x$. Then as $k\rightarrow\infty$ in $\cS$, the distance between $\big(H_{\alpha^{k\bfm}x}(\Phi_{\alpha^{k\bfm}x}f_x^k z)\big)_V$ and $
\big(H_{\alpha^{k\bfm}x}(\Phi_{\alpha^{k\bfm}x}f_x^k\Phi_x^{-1}y)\big)_V$ decays at an exponential rate of $e^{\frac{\lambda_+^{ss}+\epsilon}{\kappa'}}$ or faster, which is in particular faster than $e^{\chi_-^V-\epsilon}$ by the choice of $\epsilon_0$ in \eqref{Bands3}.

Recall $(H_x)_V$ takes value in $V$ and $\chi_-^V$ is the smallest Lyapunov exponent of $\rho^\bfm$ along $\gov$. Hence if $\big(H_x(\Phi_xz)\big)_V$ and $\big(H_x(y)\big)_V$ are not equal then the left-hand side in \eqref{HVConstEq} decays to $0$ at a rate that is slower than $e^{\chi_-^V-\epsilon}$. This yields a contradiction, and shows that  $\big(H_x(\Phi_xz)\big)_V$ assumes the constant value $\big(H_x(y)\big)_V$ along the graph of $g_{x,y}$.\end{proof}

The chart maps $\Phi_x$ may be further optimized to straighten the graphs of $g_{x,y}$. 
\begin{lemma}\cite{LY85}*{\S 8.3} For all $x\in A''$, there exists a bi-Lipschitz homeomorphism $\pi_x$ between $\cW_\alpha^s(x)\cap \Phi_x\big(B_{\bR^N}(\tau l(x)^{-1})\big)$ and a subset of $\bR^{N^s}$, such that the following are true:\begin{enumerate}
\item $\pi_x(0)=0$;
\item for all $y$, the image of $\graph(g_{x,y})$ is the intersection of $\mathrm{Image}(\pi_x)$ with a hyperplane parallel to $\bR^{N^{ss}}$.
\end{enumerate}
\end{lemma}

\begin{corollary}\label{HVConst2}For all $x\in A''$, there exists an open neighborhood $B_x$ of $x$ in $\cW_\alpha^s(x)$ and a bi-Lipschitz homeomorphism $P_x$ from an open neighborhood of  $0\in\bR^{N^s}$ to $B_x$, that sends $0$ to $x$ and satisfies that $(H_x)_V\circ P_x$ is constant along hyperplanes parallel to $\bR^{N^{ss}}$.\end{corollary}
\begin{proof}It suffices to take $B_x=\cW_\alpha^s(x)\cap \Phi_x\big(B_{\bR^N}(\tau l(x)^{-1})\big)$, and $P_x=\Phi_x\circ\pi_x^{-1}$. The claim follows from Lemma \ref{HVConst}.\end{proof}

We now show two statements that contradict each other.

\begin{lemma}\label{Jacobian0}Assume $A$ is non-empty and let $\nu$, $A''$ be as above. For every $x\in A''$, there exists a decreasing sequence of bounded open neighborhoods $B_{k,x}\subset \cW_\alpha^s(x)$ of $x$ such that $$\lim_{k\rightarrow\infty}\frac{\Vol_{\cW_\rho^s(x)}\big(H(B_{k,x})\big)}{\Vol_{\cW_\alpha^s(x)}(B_{k,x})}=0,$$
where the numerator and denominator are respectively volume forms of the induced Riemmannian metrics on $\cW_\rho^s(x)$ and $\cW_\alpha^s(x)$.\end{lemma}

\begin{proof} Fix $x$ and $\delta_0>0$ such that $B_{\bR^{N^{ss}}}(\delta_0)\times B_{\bR^{N^V}}(\delta_0)\subset P_x^{-1}B_x$. Choose $B_{k,x}\in B_x$ by $$B_{k,x}=P_x\Big(B_{\bR^{N^{ss}}}(\delta_0)\times B_{\bR^{N^V}}(\delta_k)\Big)$$ where $\delta_k<\delta_0$ and decays towards $0$ as $k\rightarrow\infty$.

As $P_x$ is bi-Lipschitz, it suffices to show that 
$$\lim_{k\rightarrow\infty}\frac{\Vol_{\cW_\rho^s\big(H(x)\big)}\Big((H\circ P_x).\big(B_{\bR^{N^{ss}}}(\delta_0)\times B_{\bR^{N^v}}(\delta_k)\big)\Big)}{\Vol_{\bR^{N^s}}\Big(B_{\bR^{N^{ss}}}(\delta_0)\times B_{\bR^{N^V}}(\delta_k)\Big)}=0,$$

The denominator is of order $O(\delta_k^{N^V})$. Identify the stable leaf $\cW_\rho^s\big(H(x)\big)$ with $G^s$ by the correspondence $H(y)\mapsto H_x(y)$, which is a local diffeomorphism. Hence our task is reduced to proving:
\begin{equation}\label{ChartBoxEq1}\Vol_{G^s}\Big((H_x\circ P_x).\big(B_{\bR^{N^{ss}}}(\delta_0)\times B_{\bR^{N^v}}(\delta_k)\big)\Big)=o(\delta_k^{N^V})\text{ as }k\rightarrow\infty,\end{equation} with $\Vol_{G^s}$ being the Haar measure on the nilpotent Lie group $G^s$.

Use Corollary \ref{SVU} to decompose $G^s$ as $G^{ss}\cdot V$. Recall that $V$ normalizes $G^{ss}$ by Lemma \ref{SUsubalg}; moreover, since it happens in a nilpotent Lie group, this normalization is unimodular, and thus $$\di\Vol_{G^s}=\di\Vol_{G^{ss}}\cdot \di\Vol_V.$$

As $\delta_k<\delta_0$, the $G^{ss}$-projection of the image set in \eqref{ChartBoxEq1} is contained in a fixed bounded subset of $G^{ss}$. So it suffices to show that 
\begin{equation}\label{ChartBoxEq2}\Vol_V\Big(\big((H_x)_V\circ P_x\big).\big(B_{\bR^{N^{ss}}}(\delta_0)\times B_{\bR^{N^v}}(\delta_k)\big)\Big)=o(\delta_k^{N^V}),\end{equation}

By Corollary \ref{HVConst}, $B_{\bR^{N^{ss}}}(\delta_0)\times B_{\bR^{N^v}}(\delta_k)$ and $B_{\bR^{N^{ss}}}(\delta_k)\times B_{\bR^{N^v}}(\delta_k)$ have the same image under $(H_x)_V\circ P_x$, since the value depends only on the second coordinate. Furthermore, because $P_x$ is Lipschitz continuous, $P_x\big(B_{\bR^{N^{ss}}}(\delta_k)\times B_{\bR^{N^v}}(\delta_k)\big)$ is contained in the ball $B_{\cW_\alpha^s(x)}(x,C\delta_k)$ for some constant $C=C(x)$. 

Therefore, it is now enough to verify
\begin{equation}\label{ChartBoxEq3}\Vol_V\Big((H_x)_V\big(B_{\cW_\alpha^s(x)}(x,\delta)\big)\Big)=o(\delta^{N^V})\text{ as }\delta\rightarrow 0,\end{equation} where $B_{\cW_\alpha^s(x)}(x,\delta)$ is the ball of radius $\delta$ surrounding $x$ in its $\cW_\alpha^s$ leaf. This follows from our hypothesis that at $x\in A''\subset A$, the differential $D_{\cW_\alpha^s}|_{y=x}\big(H_x(y)\big)_V:E_\alpha^s(x)\mapsto\gov$ has rank less than $N^V=\dim\gov$.
\end{proof}

\begin{lemma}\label{JacobianPos}For every $x\in M$, there exists a positive continuous function $J_x$ such that 
$$H_*\di\Vol_{\cW_\alpha^s(x)}=J_x\di\Vol_{\cW_\rho^s(H(x))}.$$
\end{lemma}
\begin{proof}
Since the measure $\mu$ is absolutely continuos, we have that for almost every $x$, the Radon-Nykodim derivative of its conditional measure along the stable foliation is:

$$
\frac{\di\mu^s_x}{\di\Vol _{\cW_\alpha^s(x)}}(y)=r_x(y):=\prod_{k\geq 0}\frac{J^s\alpha^\bfm(\alpha^{k\bfm}(y))}{J^s\alpha^\bfm(\alpha^{k\bfm}(x))}
$$
for $y\in \cW_\alpha^s(x)$, (remember that the conditional measure is only well determined up to multiplication by a constant depending on $x$) see for instance \cite{PS83}. Here $J^s\alpha^\bfm(y)$ is the Jacobian of $\alpha^\bfm$ along the  stable bundle. 

Let us take for simplicity some small neighborhood chart $B$ in such a way that the stable foliation locally trivializes in this neighborhood and restrict all measures to this neighborhood chart so that conditional measures become finite measures.

Observe that $x\to r_x$ is a H\"older continuous function, moreover there is a constant $c>0$ such that $r_x\geq c$ for every $y$. This follows from the fact that the stable Jacobian is a H\"older continuous function plus uniform contraction along the stable manifold, from which we get that the product is uniformly convergent, and that $r_x(y)$ is equal to $\frac 1{r_y(x)}$ and hence bounded away from zero and infinity. Therefore we get that the measurable mapping $x\to\mu^s_x$ coincides with a continuous mapping $$x\to\bar\mu^s_x:=r_x\di\Vol_{\cW_\alpha^s(x)}.$$

Since $H$ sends $\mu$ to Lebesgue measure $Leb$, it sends $\mu_x^s$ to $Leb_{H(x)}^s$ for almost every $x$, i.e. \begin{eqnarray}\label{equalityconditionals}
H_*\mu_x^s=Leb_{H(x)}^s=\di\Vol_{\cW_\rho^s(H(x))}.\end{eqnarray}

The right hand side of equation (\ref{equalityconditionals}) is a continuous mapping $x\to \di\Vol_{\cW_\rho^s(H(x))}$ from $B$ to finite measures on $H(B)$. The left hand side is a priory only measurable, but since $H_*$ is a continuous operator on finite measures and $\mu_x^s$ coincides with $\bar\mu_x^s$ a.e. we get that $$H_*\bar\mu_x^s=\di\Vol_{\cW_\rho^s(H(x))}.$$

Hence defining $$J_x(z)= \frac{1}{r_x(H^{-1}(z))}$$ we get that $$H_*\di\Vol_{\cW_\alpha^s(x)}= H_*\left(\frac{1}{r_x}\bar\mu^s_x\right)=J_xH_*\bar\mu_x^s=J_x\di\Vol_{\cW_\rho^s(H(x))},$$ for every $x$. This is because both sides are continuous and coincide on a full measure subset. The proof is hence completed.\end{proof}

\begin{proof}[Proof of Proposition \ref{ssSmooth}]Lemma \ref{JacobianPos} is obviously incompatible with Lemma \ref{Jacobian0}. So the only possibility is  that the main hypothesis we made, that $A$ is non-empty, fails and Lemma \ref{Jacobian0} does not apply. We immediately deduce the Proposition \ref{ssSmooth} from this and Corollary \ref{RegSmooth}.\end{proof}

\subsection{Crossing the Weyl chamber wall}

We now construct new Anosov elements in the adjacent Weyl chamber. As always, we are assuming $\alpha^{\bfn_0}$ is Anosov and $\cC_0$ is the Weyl chamber containing $\bfn_0$.

\begin{proposition}\label{NewAnosov} For any Weyl chamber $\cC$ adjacent to $\cC_0$, for all elements $\bfn\in\cC$, $\alpha^\bfn$ is Anosov.\end{proposition}

We verify the Anosov property using the following criterion by R. Ma\~n\'e.

\begin{theorem}\label{Mane}\cite{M77} A diffeomorphism $f$ of a smooth manifold $M$ is Anosov if and only if the dimensions of stable manifolds at all periodic points are the same and $f$ is quasi-Anosov, that is, $\{|Df^iv|\}_{i=-\infty}^\infty$ is unbounded for all non-zero $v\in TM$.\end{theorem}

\begin{proof}[Proof of Proposition \ref{NewAnosov}] Let $\chi$ be a Lyapunov exponent such that the wall between $\cC_0$ and $\cC$ is a cone in $\ker\chi$. Since $\alpha^\bfn$ is Anosov if and only if $\alpha^{-\bfn}$ is, as before we may assume $\chi(\bfn_0)<0$ without loss of generality. Let $\gov=\gov^{[\chi]}$. There may be Lyapunov exponents negatively proportional to $\chi$ in the Lyapunov decomposition, in which case there is a non-trivial coarse Lyapunov subspace $\gov^{[-\chi]}$. We write always $\gov'=V^{[-\chi]}$ even when it is trivial. The corresponding coarse Lyapunov subgroups are $V\subset G^s$ and $V'\subset G^u$.

We have been working with the $G^{ss}\cdot V\cdot G^u$ introduced in \S\ref{secSVU}. Symmetrically, by taking $\alpha^{-\bfn_0}$ instead of $\alpha^{\bfn_0}$, there is a strong unstable-weak unstable splitting of $\gog^u$ into $\gog^{uu}\oplus\gov'$. And $G$ can be decomposed as $G^{uu}\cdot V'\cdot G^s$. For $a\in G$, let $a_V$ be its $V$ component in the $G^{ss}\cdot V\cdot G^u$ decomposition, and $a_{V'}$ be its $V'$ component in the $G^{uu}\cdot V'\cdot G^s$ decomposition. 

By Proposition \ref{ssSmooth}, there are continuous strong stable and strong unstable foliations $\cW_\alpha^{ss}\subset \cW_\alpha^s$ and $\cW_\alpha^{uu}\subset\cW_\alpha^u$ of $\cM$ into smooth leaves. (When $\gov'$ is trivial, $\cW_\alpha^{uu}$ coincide with $\cW_\alpha^u$.)

We now check $\alpha^\bfn$ is quasi-Anosov. For any $x\in M$ and any non-zero $v\in T_xM$. $v$ splits as $v=v_s+v_u$ in $T_xM=E_\alpha^s(x)\oplus E_\alpha^u(x)$. Since $\alpha^\bfn$ preserves the splitting $E_\alpha^s\oplus E_\alpha^u$, it suffices to show that one of $v_s$ and $v_u$ become unbounded under $\{D\alpha^{i\bfn}\}_{i=-\infty}^\infty$. Without loss of generality, we assume $v_s\neq 0$. The $v^u\neq 0$ case can be solved in a symmetric manner.

\noindent{\bf Case 1.} Suppose first $v_s$ is in $E_\alpha^{ss}$, the distribution tangent to the subfoliation $\cW_\alpha^{ss}$. Recall by definition \eqref{ssFoliation}, $H$ sends $\cW_\alpha^{ss}$ into the strong stable foliation $\cW_\rho^{ss}$ of $M$ with respect to the affine action $\rho$, which consists of orbits of $G^{ss}$. 

For any Lyapunov subspace $\gov^{\chi'}\subset\gog^{ss}$, $\ker\chi'$ is different from $\ker\chi$. Since $\ker\chi$ is the only Weyl chamber wall separating $\cC$ and $\cC_0$, $\chi'(\bfn)<0$. As the tangent space $E_\rho^{ss}$ of $\cW_\rho^{ss}$ is identified with $\gog^{ss}$, we see that all Lyapunov exponents of the algebraic action $\rho:\bZ^r\curvearrowright G$ along $\cW_\rho^{ss}$ are negative.

Lemma \ref{Subadd2}, applied to the conjugacy $H$ between $\alpha^\bfn$ and $\rho^\bfn$, implies that there exists $k\in\bN$ such that $\|D\alpha^{k\bfn_1}|_{E_\alpha^{ss}}\|<1$, thus  $\|(D\alpha^{-ik\bfn})v_s\|\rightarrow\infty$ as $i\rightarrow\infty$.

\noindent{\bf Case 2.} Suppose instead  $v_s\notin E_\alpha^{ss}$. By Lemma \ref{SSLocal} and the everywhere regularity of $D_{\cW_\alpha^s}|_{y=x}\big(H_x(y)\big)_V$, this means the vector $w=\Big(D_{\cW_\alpha^s}|_{y=x}\big(H_x(y)\big)_V\Big)v_s$ is non-zero. By \eqref{DerAct},
\begin{equation}\label{NewAnosovEq1}\rho^{-i\bfn}|_\gov\left(D_{\cW_\alpha^s}|_{y=\alpha^{i\bfn}x}\big(H_{\alpha^{i\bfn}x}(y)\big)_V\right)\big((D_x\alpha^{i\bfn})|_{E_\alpha^s}\big)v_s=w.\end{equation}

Because $\ker\chi$ separates $\cC$ from $\cC_0$ and $\chi(\bfn_0)<0$, $\chi(\bfn)$ is positive. Thus $\|\rho^{-i\bfn}|_\gov\|$ decays exponentially fast as $i\rightarrow\infty$. On the other hand,  $D_{\cW_\alpha^s}|_{y=\alpha^{i\bfn}x}\big(H_{\alpha^{i\bfn}x}(y)\big)_V$ has bounded norm as $D_{\cW_\alpha^s}|_{y=x}\big(H_x(y)\big)_V$ is continuous on the compact manifold $M$ by Proposition \ref{HVSmooth}. Therefore, in order for \eqref{NewAnosovEq1} to hold, the size of the vector $(D_x\alpha^{i\bfn})v_s=\big((D_x\alpha^{i\bfn})|_{E_\alpha^s}\big)v_s$ must grow exponentially as $i\rightarrow\infty$.

Therefore in both cases, $(D_x\alpha^{i\bfn})v_s$ is unbounded as $i$ ranges over $\bZ$.  This shows $\alpha^\bfn$ is quasi-Anosov.

On the other hand, because $\alpha^\bfn$ is continuously conjugate to the algebraic action $\rho^\bfn$, stable manifolds of $\alpha^\bfn$ at all points must have the same dimension as $\gog^s$. By Theorem \ref{Mane}, $\alpha^\bfn$ is Anosov.
\end{proof}
 
\section{Conclusion of the proof}\label{Conclusion}\label{secConc}

\begin{proof}[Proof of Theorem \ref{Main}]By lifting $\alpha$, $\rho$ and $H$ from an infranilmanifold to a finite cover if necessary \cite{M74}*{Remark 2}, it can be assumed that $M$ is a nilmanifold. 

\noindent{\bf Standard nilmanifolds.} Suppose $M$ is a compact infranilmanifold equipped with the standard smooth structure. Given the fact that at least one $\alpha^{\bfn_0}$ is Anosov, Proposition \ref{NewAnosov} allows us to start from the initial Weyl chamber $\cC_0$ containing $\bfn_0$, cross Weyl chamber walls and finally reach all Weyl chambers. So we know that for any element $\bfn$ in the interior of any Weyl chamber, $\alpha^\bfn$ is Anosov.

Applying Fisher-Kalinin-Spatzier's theorem \cite{FKS13}*{Theorem 1.3}, one concludes that the conjugacy $H$ is a diffeomorphism. We emphasize that the proof actually works for all possible H\"older conjugacies, from the homotopy class of identity, between $\alpha$ and its linearization $\rho$, rather than a particular Franks-Manning conjugacy. 

\noindent{\bf Exotic nilmanifolds.} Now we treat the case where $M$ has an exotic differential structure $\omega$.

In this case there is a Franks-Manning conjugacy $H:M\mapsto M$ that interwines $\alpha$ with its linearization $\rho$. $\alpha$ is smooth with respect to $\omega$, while $\rho$, consisting of affine automorphisms, is smooth with respect to the standard differential structure $\omega_0$. We wish to show $H: (M,\omega)\mapsto (M,\omega_0)$ is $C^\infty$, which in particular implies $(M,\omega)$ is a standard structured nilmanifold, and the exotic case actually does not happen.

When $\dim M\geq 5$. It was proved by Davis in the appendix of \cite{FKS13} that $M$ is finitely covered by a nilmanifold $M_1$ with standard differential structure. By Lemma \ref{NilLift} below, we can reduce to the standard case above by lifting to a finite cover $M_2$ of $M_1$,

It remains to treat the case of $\dim M\leq 5$. This actually cannot happen in dimensions 2 and 3 \cites{R25,M52}.

Assume $M$ is 4-dimensional\footnote{In fact, this forces $M$ to be a torus.}. Observe the diagonal $\bZ^r$-actions $\alpha\times\alpha$ and $\rho\times\rho$ on  $M\times M$ are topologically conjugate by $H\times H$. It is easy to verify: $\alpha^{\bfn_0}\times\alpha^{\bfn_0}$ is Anosov, $\rho\times\rho$ has no rank-one factor, and $H\times H$ is a H\"older conjugacy homotopic to identity.
 
Because $M\times M$ is an infranilmanifold of dimension $8\geq 5$, $H\times H$ is a diffeomorphism, which can happen only if $H$ is a diffeomorphism itself. The proof is complete.\end{proof}

\begin{lemma}\label{NilLift} If $M_1\mapsto M$ is a finite covering map, then there is a compact nilmanifold $M_2$ that finitely covers $M_1$ and hence $M$ such that any homeomorphism $f:M\mapsto M$ lifts to a homeomorphism $\tilde f: M_2\mapsto M_2$.\end{lemma}

\begin{proof}$M$ and $M_1$ are quotients of the same simply connected nilpotent Lie group $G$, respectively by lattices $\Gamma=\pi_1(M)$ and $\Gamma_1=\pi_1(M_1)\subset\pi_1(M)$.  $f$ induces an automorphism $f_*$ on the fundamental group $\Gamma$. As $M_1\mapsto M$ is a finite cover, the index $I$ of $\Gamma_1$ in $\Gamma$ is finite. Then $f_*(\Gamma)$ has index $I$ in $\Gamma$. As  $\Gamma$ is finitely generated,  it has only finitely many subgroups of given index $I$. Hence $\Gamma_2=\bigcap_f f_*(\Gamma_1)$, where the summation is taken over all homeomorphisms, has finite index in $\Gamma$ and is invariant under all the $f_*$'s. It follows that any $f$ can be lifted to the compact nilmanifold $M_2=G/\Gamma_2$.
\end{proof}

\appendix
\section{A regularity theorem along H\"older foliations}\label{secSobolev}

Let $M$ be a compact Riemannian manifold with a continuous foliation $\cF$ with smooth leaves. Write points in $\bR^{\dim M}$ as $(x,y)\in\bR^{\dim M-\dim\cF}\times\bR^{\dim\cF}$. Suppose there are H\"older continuous local chart maps $\Gamma$  from open sets $O\subset\bR^{\dim M}$ to $M$, such that $\Gamma$ sends subspaces parallel to the $y$-hyperplane to leaves of $\cF$ by $C^\infty$ local immersions, and pushes the Lebesgue measure on $\bR^{\dim M}$ to an absolutely continuous measure $J\di\Vol$ on $\Gamma(O)$. Moreover, assume that $J$ is also $C^\infty$ along the $\cF$ leaves, and that all partial derivatives of $\Gamma$ and $J$ of arbitrary orders along the $y$ direction are H\"older continuous on $O$.  

\begin{theorem}\label{RauchTaylor}Suppose $M$ and $\cF$ are as above. If a H\"older continuous function $\phi$ on $M$ satisfies that for all $\theta>0$, all partial derivatives of $\phi$ along $\cF$ of all orders belong to $(C^\theta)^*$, then all these partial derivatives are H\"older continuous.\end{theorem}

This theorem is motivated by \cite{RT05}*{Theorem 1.1}.

Let $B$ be the closed unit ball in $\bR^m$. Given $\alpha\in [0,1)$ let us denote with $C_0^{\alpha}(B\times\bT^l)$ the closure of the $C^\infty$ functions of compact support inside the interior of the unit ball times $\bT^l$ w.r.t. the $C^\alpha$-norm $\|\phi\|_{C^\alpha}$ given in \eqref{HolderNorm}. When the space $B\times\bT^l$ is understood we shall denote it directly with $C_0^{\alpha}$. Denote points in $B\times\bT^l$ by $(x,y)$. Take $\alpha\in [0,1)$. Given $\phi\in C_0^\alpha$, $k\in \bN$ and $\gamma\in [0,1)$, we say that $\phi$ has $k$-th derivative in the $y$ direction belonging to $(C^{\gamma})^*$ if for any multiindex ${\bf r}=(r_1,\dots r_l)$ with $|{\bf r}|=k$ we have that  $$\partial^{\bf r}_y \phi\in (C^{\gamma})^*$$ in the sense that, there is a constant $C=C(\gamma,{\bf r})$ such that for any $u\in C^{\infty}_c(\bR^m\times \bT^l)$, we have that $$\left|\int_{\bR^m\times\bT^l}\phi\partial^{\bf r}_y u\di x\di y\right|\leq C\|u\|_{C^{\gamma}}.$$ Here $|\bfr|=r_1+r_2+\cdots+r_l$, and $\partial_y^\bfr=\partial_{y_1}^{r_1}\cdots\partial_{y_l}^{r_l}$.

\begin{proposition}\label{RauchTaylor2}
Let $\alpha\in(0,1)$ and let $\phi\in C_0^\alpha$.  Let us assume that  there is $\gamma\in[0,1)$ such that partial derivatives of all orders in the $y$ direction of $\phi$ belong to $(C^{\gamma})^*$. Then for some $\beta\in (0,\alpha)$ and for any multiindex ${\bf r}=(r_1,\dots r_l)$ we have that  $$\partial^{\bf r}_y \phi\in C_0^{\beta}.$$
\end{proposition}
That is, if $\phi$ is H\"older and has weak derivatives of all orders in the $y$ direction, then it has H\"older continuous derivatives in the $y$ direction.

Theorem \ref{RauchTaylor} is an almost immediate corollary to Proposition \ref{RauchTaylor2}.

\begin{proof}[Proof of implication Proposition \ref{RauchTaylor2}$\Rightarrow$Theorem \ref{RauchTaylor}] By hypothesis on the foliation $\cF$ and making partition of unity, we may assume $M=\bR^m\times\bR^l$, $\cF$ is the foliation into $\bR^l$ hyperplanes in $y$ direction, and $\phi$ is compactly supported. We can even assume $\phi$ is supported on $B\times(-\frac14,\frac14)^l$. This new case follows from Proposition \ref{RauchTaylor2}. by identify $(-\frac14,\frac14)^l$ with an open subset of $\bT^l$. 
\end{proof}

We now aim to prove Proposition \ref{RauchTaylor2}. Consider $\phi$ as in Proposition \ref{RauchTaylor2}. Let us write $$\phi(x,y)=\sum_{{\bf n}\in\bZ^l}\phi_{\bf n}(x)e^{2\pi i {\bf n}\cdot y}$$
where $$\phi_{\bf n}(x)=\int_{\bT^l}\phi(x,y)e^{-2\pi i {\bf n}\cdot y}\di y.$$

We have that for any ${\bf n}\in\bZ^l$, \begin{eqnarray}\label{holder}\|\phi_{\bf n}\|_{C^\alpha(\bR^m)}\leq \|\phi\|_{C^\alpha(\bR^m\times \bT^l)}.\end{eqnarray}

On the other hand, since $\phi$ has partial derivatives of all orders in the $y$ direction in the weak $C^{\gamma}$-sense,  for any $\bfr$ there is a constant $C_0(\bfr)$ such that for any  $u\in C^{\infty}_c(B\times \bT^l)$, \begin{equation}\label{DerHolderStar}\left|\int_{B\times\bT^l}\phi\partial^\bfr_yu\di x\di y\right|\leq C_0(\bfr)\|u\|_{C^{\gamma}_0(B\times \bT^l)}.\end{equation}

In the sequel we shall use several standard embeddings of Sobolev's spaces and interpolation. Given $s\in\bR$ and $1\leq p<\infty$ let  $H^{s,p}=H^{s,p}(\bR^m)$ be the Sobolev space of order $s$ over $L^p$, i.e. $H^{s,p}(\bR^m)$ is the closure of  $C_c^\infty(\bR^m)$  w.r.t. the norm $$\|u\|_{s,p}=\left\|\left((1+|\xi|^2)^{s/2}|\hat u|\right)^{\vee}\right\|_{L^p}.$$
Here $\hat u$ stands for the Fourier transform and $u^\vee$ is the inverse transform. 
The following can be found in any book on Sobolev spaces (see for instance \cite{RS96, T78}).
\begin{lemma}\label{SobolevEmbedding} Let $s,t\in\bR$, for $1<p<\infty$, let $p^*$ be such that $1=\frac{1}{p}+\frac{1}{p^*}$. 
\begin{enumerate}
\item If $s<t$ then $$H^{s,p}\subset H^{t,p}$$ and the embedding is continuous.
\item $$\left(H^{s,p}\right)^*=H^{-s,p^*}.$$
\item If $\frac{1}{q}=\frac{1}{p}-\frac{s-t}{m}$ then $$H^{s,p}(\bR^m)\subset H^{t,q}(\bR^m)$$ where the embedding is continuous.
\item If $\alpha\in(0,1)$, $r\in\bN$ and $\frac{s-r-\alpha}{m} = \frac{1}{p}$ then $$H^{s,p}(\bR^m)\subset C^{r,\alpha}(\bR^m)$$ and the embedding is continuous.
\item\label{HolderIntoSobo} For any $\alpha\in (0,1)$ and for any $p\geq 1$, and for any $s<\alpha$, $$C^{\alpha}(\bR^m)\cap L^p(\bR^m)\subset H^{s,p}(\bR^m)$$ and the embedding is continuous. 
\item\label{interpolation} If $s_1,s_2\in\bR$, $1< p_1,p_2<\infty $, $\theta\in[0,1]$ and $t=\theta s_1+(1-\theta)s_2$, $\frac{1}{p}=\frac{\theta}{p_1}+\frac{1-\theta}{p_2}$ then there is a constant $C=C(s_1,s_2,\theta,p_1,p_2)$ such that $$|u|_{t,p}\leq C |u|_{s_1,p_1}^{\theta}|u|_{s_2,p_2}^{1-\theta}.$$
\end{enumerate}
\end{lemma}

\begin{lemma}\label{weakdecay}
For all multiindex $\bfr$, there is a constant $C(\bfr)$ such that,  for any $\bfn\in\bZ^l$, with $|\bfn^{\bfr}|\neq 0$,

\begin{equation}\label{weakholder}\|\phi_{\bfn}\|_{(C^\gamma)^*}\leq \frac{C(\bfr)|\bfn|}{|\bfn^{\bfr}|}.\end{equation}

\end{lemma}
Here and below, $|\bfz|=\sum_{i=1}^l|z_i|$ and $\bfz^\bfr=z_1^{r_1}z_2^{r_2}\cdots z_l^{r_l}$ for all $\bfz\in\bC^l$.

\begin{proof} Take $w\in C_c^\infty(B)$.
\begin{equation*}\label{weakdecayEq1}\begin{split}\displaystyle&\left| \int_B\phi_\bfn(x) w(x)\di x\right|\\
=&\left|\int_{B\times\bT^l}\phi(x,y)\cdot e^{2\pi i\bfn\cdot y}w(x)\di x\di y\right|\\
=&\left|(2\pi i\bfn)^{-\bfr}\int_{B\times\bT^l}\phi(x,y)\partial_y^\bfr \left(e^{2\pi i\bfn\cdot y}w(x)\right)\di x\di y\right|
\end{split}\end{equation*}

By \eqref{DerHolderStar},
\begin{equation*}\begin{split}\displaystyle\left| \int_B\phi_\bfn(x) w(x)\di x\right|\leq &\frac1{(2\pi)^{|\bfr|}|\bfn^\bfr|}\cdot\left\|e^{2\pi i\bfn\cdot y}w(x)\right\|_{C^\gamma}\\
\leq &\frac{C_0(\bfr)}{(2\pi)^{|\bfr|}|\bfn^\bfr|}\cdot\left\|e^{2\pi i\bfn\cdot y}\right\|_{C^\gamma}\|w\|_{C^\gamma}\\
\leq &\frac{C_0(\bfr)}{(2\pi)^{|\bfr|}|\bfn^\bfr|}\cdot\left\|e^{2\pi i\bfn\cdot y}\right\|_{C^1}\|w\|_{C^\gamma}\\
=&\frac{2\pi |\bfn|C_0(\bfr)}{(2\pi)^{|\bfr|}|\bfn^\bfr|}\cdot \|w\|_{C^\gamma},
\end{split}\end{equation*}
which is the lemma.
\end{proof}

\begin{lemma}\label{strongdecay}
There exist $\beta>0$ and $\theta>0$ such that for any multiindex $\bfr$, there is a constant $C(\bfr,\beta)$ such that for any $\bfn\in\bZ^l$, with $|\bfn^{\bfr}|\neq 0$,
\begin{eqnarray}\label{strongholder}\|\phi_{\bfn}\|_{C^\beta}\leq \frac{C(\bfr,\beta)|\bfn|^\theta}{|\bfn^{\bfr}|^{\theta}}.\end{eqnarray}
\end{lemma}

\begin{proof} Fix a sufficiently large $p^*$ such that $\frac m{p^*}<\frac\alpha 4$.

Take $s$ and $p$, such that $\frac{s-\gamma}m=\frac1p$ and $1=\frac1p+\frac1{p^*}$. By Lemma \ref{SobolevEmbedding}, we have a continuous embeddings $(C^\gamma)^*\subset(H^{s,p})^*=H^{-s,p^*}$. Therefore by Lemma \ref{weakdecay}, there is $C_1(\bfr)$ such that
\begin{equation}\label{strongdecayEq1}\|\phi_\bfn\|_{H^{-s,p^*}}\leq\frac{C_1(\bfr)|\bfn|}{|\bfn^\bfr|}.\end{equation}

On the other hand, $\phi_\bfn$ is clearly bounded by $\|\phi\|_{L^\infty}\leq\|\phi\|_{C^\alpha}$. Hence $\|\phi_\bfn\|_{L^{p^*}}\leq a_1\|\phi\|_{C^\alpha}$ where $a_1$ depends only on the volume of the unit ball $B$. Combining this with \eqref{holder}, we know by Lemma \ref{strongdecay}.\eqref{HolderIntoSobo} that there is a constant $C_2$ that depends only on the dimension, such that
\begin{equation}\label{strongdecayEq2}\|\phi_\bfn\|_{H^{\frac{\alpha}{2},p^*}}\leq C_2 \|\phi\|_{C^\alpha}.\end{equation}

Choose $\theta\in(0,1)$ such that $\theta\cdot (-s)+(1-\theta)\cdot\frac\alpha 2=\frac\alpha 4$. The interpolation formula Lemma \ref{SobolevEmbedding}.\eqref{interpolation} allows to merge \eqref{strongdecayEq1} and \eqref{strongdecayEq2} into:
\begin{equation}\label{strongdecayEq3}\|\phi_\bfn\|_{H^{\frac\alpha 4,p^*}}\leq C_1(\bfr)^\theta C_2^{1-\theta}\|\phi\|_{C^\alpha}^{1-\theta}\left(\frac{|\bfn|}{|\bfn^\bfr|}\right)^\theta.\end{equation}
Take  $\beta>0$ such that $\frac{\frac\alpha 4-\beta}m=\frac{1}{p^*}$, which is possible thanks to the choice of $p^*$. \eqref{strongdecayEq3} establishes the lemma since $H^{\frac\alpha 4,p^*}$ continuously embeds into $C^\beta$.\end{proof}

\begin{corollary}\label{strongholderdecay}
For any $T>0$ there is $C(T)>0$ such that $$\|\phi_{\bfn}\|_{C^\beta}\leq \frac{C(T)}{|\bfn|^T}$$ for any $0\neq \bfn\in\bZ^l$.
\end{corollary}

\begin{proof} Given $\bfn\neq 0$, there is $1\leq i\leq l$ such that  $|n_i|\geq \frac{|\bfn|}l$. Choose $\bfr$ by letting $r_i=\lceil \frac{T+1}\theta\rceil$ and $r_j=0$ if $j\neq i$. Then the corollary directly follows from Lemma \ref{strongdecay}.
\end{proof}

\begin{proof}[Proof of Proposition \ref{RauchTaylor2}]

Let $A>0$ and $\bfr$ be a multiindex with $|\bfr|=r$, define $$\phi_{\bfr,A}(x,y)=\sum_{\bfn\in\bZ^l, |\bfn|\leq A}(2\pi i)^r\bfn^\bfr\phi_\bfn(x)e^{2\pi i\bfn\cdot y}.$$ Then there are constants $C_1(\bfr), C_2(r)$ such that  
\begin{equation*}\begin{split} 
&|\phi_{\bfr,A}(x,y)-\phi_{\bfr,A}(z,y)|\\
\leq &C_1(\bfr)\sum_{\bfn\in\bZ^l, |\bfn|\leq A}|\bfn|^r|\phi_\bfn(x)-\phi_\bfn(z)|\\
\leq &C_1(\bfr)\sum_{\bfn\in\bZ^l, |\bfn|\leq A}|\bfn|^r\|\phi_\bfn\|_{C^{\beta}_0}\dist(x,z)^\beta\\
\leq &C_1(\bfr)C(r+2l+1)\left(\sum_{\bfn\in\bZ^l, |\bfn|\leq A}\frac{1}{|\bfn|^{2l+1}}\right)\dist(x,z)^\beta\\
\leq &C_2(r)\dist(x,z)^\beta,
\end{split}\end{equation*}
where $C(r+2l+1)$ is that from Corollary \ref{strongholderdecay}.

This same computation gives that $$\lim_{A\rightarrow\infty}\phi_{\bfr,A}=\partial^\bfr_y\phi$$ in $C^{\beta}$-topology. 
Which implies that $\partial^\bfr_y\phi$ is $C^{\beta}$. 
\end{proof}


\begin{bibdiv}
\begin{biblist}

\bib{dlL01}{article}{
   author={de la Llave, Rafael},
   title={Remarks on Sobolev regularity in Anosov systems},
   journal={Ergodic Theory Dynam. Systems},
   volume={21},
   date={2001},
   number={4},
   pages={1139--1180},
}

\bib{EF07}{article}{
   author={Einsiedler, Manfred},
   author={Fisher, Travis},
   title={Differentiable rigidity for hyperbolic toral actions},
   journal={Israel J. Math.},
   volume={157},
   date={2007},
   pages={347--377},
}

\bib{FJ78}{article}{
   author={Farrell, F. T.},
   author={Jones, L. E.},
   title={Anosov diffeomorphisms constructed from $\pi _{1}\,{\rm
   Diff}\,(S^{n})$},
   journal={Topology},
   volume={17},
   date={1978},
   number={3},
   pages={273--282},
}

\bib{FKS11}{article}{
   author={Fisher, David},
   author={Kalinin, Boris},
   author={Spatzier, Ralf},
   title={Totally nonsymplectic Anosov actions on tori and nilmanifolds},
   journal={Geom. Topol.},
   volume={15},
   date={2011},
   number={1},
   pages={191--216},
}
		
\bib{FKS13}{article}{
   author={Fisher, David},
   author={Kalinin, Boris},
   author={Spatzier, Ralf},
   title={Global rigidity of higher rank Anosov actions on tori and
   nilmanifolds,  \rm{with an appendix by J. Davis}},
   journal={J. Amer. Math. Soc.},
   volume={26},
   date={2013},
   number={1},
   pages={167--198}
}

\bib{F69}{article}{
   author={Franks, John},
   title={Anosov diffeomorphisms on tori},
   journal={Trans. Amer. Math. Soc.},
   volume={145},
   date={1969},
   pages={117--124},
}

\bib{G10}{article}{
   author={Gogolev, Andrey},
   title={Diffeomorphisms H\"older conjugate to Anosov diffeomorphisms},
   journal={Ergodic Theory Dynam. Systems},
   volume={30},
   date={2010},
   number={2},
   pages={441--456},
}

\bib{G07}{article}{
   author={Gorodnik, Alexander},
   title={Open problems in dynamics and related fields},
   journal={J. Mod. Dyn.},
   volume={1},
   date={2007},
   number={1},
   pages={1--35},
}

\bib{GS}{article}{
   author={Gorodnik, A},
   author={Spatzier, R},
   title={Mixing properties of $\bZ^k$-actions on nilmanifolds},
   journal={preprint},
}		

\bib{KK99}{article}{
   author={Kalinin, Boris},
   author={Katok, Anatole},
   title={Invariant measures for actions of higher rank abelian groups},
   conference={
      title={Smooth ergodic theory and its applications},
      address={Seattle, WA},
      date={1999},
   },
   book={
      series={Proc. Sympos. Pure Math.},
      volume={69},
      publisher={Amer. Math. Soc.},
      place={Providence, RI},
   },
   date={2001},
   pages={593--637}
}

\bib{KSad06}{article}{
   author={Kalinin, Boris},
   author={Sadovskaya, Victoria},
   title={Global rigidity for totally nonsymplectic Anosov $\Bbb Z^k$
   actions},
   journal={Geom. Topol.},
   volume={10},
   date={2006},
   pages={929--954 (electronic)},
}

\bib{KSad07}{article}{
   author={Kalinin, Boris},
   author={Sadovskaya, Victoria},
   title={On the classification of resonance-free Anosov $\Bbb Z^k$
   actions},
   journal={Michigan Math. J.},
   volume={55},
   date={2007},
   number={3},
   pages={651--670},
}

\bib{KS07}{article}{
   author={Kalinin, Boris},
   author={Spatzier, Ralf},
   title={On the classification of Cartan actions},
   journal={Geom. Funct. Anal.},
   volume={17},
   date={2007},
   number={2},
   pages={468--490},
}

\bib{KKS02}{article}{
   author={Katok, Anatole},
   author={Katok, Svetlana},
   author={Schmidt, Klaus},
   title={Rigidity of measurable structure for ${\Bbb Z}^d$-actions by
   automorphisms of a torus},
   journal={Comment. Math. Helv.},
   volume={77},
   date={2002},
   number={4},
   pages={718--745}
}

\bib{KL91}{article}{
   author={Katok, A.},
   author={Lewis, J.},
   title={Local rigidity for certain groups of toral automorphisms},
   journal={Israel J. Math.},
   volume={75},
   date={1991},
   number={2-3},
   pages={203--241},
}

\bib{KS94}{article}{
   author={Katok, Anatole},
   author={Spatzier, Ralf J.},
   title={First cohomology of Anosov actions of higher rank abelian groups
   and applications to rigidity},
   journal={Inst. Hautes \'Etudes Sci. Publ. Math.},
   number={79},
   date={1994},
   pages={131--156},
   issn={0073-8301},
   review={\MR{1307298 (96c:58132)}},
}

\bib{KS97}{article}{
   author={Katok, A.},
   author={Spatzier, R. J.},
   title={Differential rigidity of Anosov actions of higher rank abelian
   groups and algebraic lattice actions},
   journal={Tr. Mat. Inst. Steklova},
   volume={216},
   date={1997},
   number={Din. Sist. i Smezhnye Vopr.},
   pages={292--319},
   issn={0371-9685},
   translation={
      journal={Proc. Steklov Inst. Math.},
      date={1997},
      number={1 (216)},
      pages={287--314},
      issn={0081-5438},
   },
}

\bib{LY85}{article}{
author={Ledrappier, F.},
   author={Young, L.-S.},
   title={The metric entropy of diffeomorphisms. II. Relations between entropy, exponents and dimension},
   journal={Ann. of Math. (2)},
   volume={122},
   date={1985},
   number={3},
   pages={540--574},
}

\bib{M51}{article}{
   author={Malcev, A. I.},
   title={On a class of homogeneous spaces},
   journal={Amer. Math. Soc. Translation},
   volume={1951},
   date={1951},
   number={39},
   pages={33},
}

\bib{M78}{article}{
   author={Malyshev, F. M.},
   title={Decompositions of nilpotent Lie algebras},
   journal={Math. Notes Acad. Sci. USSR},
   volume={23},
   date={1978},
   number={1},
   pages={17--18},
}

\bib{M77}{article}{
   author={Ma{\~n}{\'e}, Ricardo},
   title={Quasi-Anosov diffeomorphisms and hyperbolic manifolds},
   journal={Trans. Amer. Math. Soc.},
   volume={229},
   date={1977},
   pages={351--370},
}

\bib{M74}{article}{
   author={Manning, Anthony},
   title={There are no new Anosov diffeomorphisms on tori},
   journal={Amer. J. Math.},
   volume={96},
   date={1974},
   pages={422--429},
}

\bib{MQ01}{article}{
   author={Margulis, Gregory A.},
   author={Qian, Nantian},
   title={Rigidity of weakly hyperbolic actions of higher real rank
   semisimple Lie groups and their lattices},
   journal={Ergodic Theory Dynam. Systems},
   volume={21},
   date={2001},
   number={1},
   pages={121--164},
}


\bib{M52}{article}{
   author={Moise, Edwin E.},
   title={Affine structures in $3$-manifolds. V. The triangulation theorem
   and Hauptvermutung},
   journal={Ann. of Math. (2)},
   volume={56},
   date={1952},
   pages={96--114},
}



\bib{P69}{article}{
   author={Parry, William},
   title={Ergodic properties of affine transformations and flows on
   nilmanifolds. },
   journal={Amer. J. Math.},
   volume={91},
   date={1969},
   pages={757--771},
}


\bib{PS83}{article}{
    AUTHOR = {Pesin, Ya. B.} 
       AUTHOR ={Sina{\u\i}, Ya. G.},
     TITLE = {Gibbs measures for partially hyperbolic attractors},
   JOURNAL = {Ergodic Theory Dynam. Systems},
  FJOURNAL = {Ergodic Theory and Dynamical Systems},
    VOLUME = {2},
      YEAR = {1982},
    NUMBER = {3-4},
     PAGES = {417--438}
}

\bib{R25}{article}{
   author={Rad\'o, T.},
   title={\"Uber den Begriff der Riemannschen Fl\"ache},
   journal={Acta Litt. Sci. Szeged},
   volume={2},
   date={1925},
   pages={101-121},
}


\bib{RT05}{article}{
AUTHOR = {Rauch, Jeffrey}
author={Taylor, Michael},
     TITLE = {Regularity of functions smooth along foliations, and elliptic
              regularity},
   JOURNAL = {J. Funct. Anal.},
  FJOURNAL = {Journal of Functional Analysis},
    VOLUME = {225},
      YEAR = {2005},
    NUMBER = {1},
     PAGES = {74--93},
      ISSN = {0022-1236},
     CODEN = {JFUAAW},
}
		
\bib{RH07}{article}{
   author={Rodriguez Hertz, Federico},
   title={Global rigidity of certain abelian actions by toral automorphisms},
   journal={J. Mod. Dyn.},
   volume={1},
   date={2007},
   number={3},
   pages={425--442},
}



\bib{R79}{article}{
   author={Ruelle, David},
   title={Ergodic theory of differentiable dynamical systems},
   journal={Inst. Hautes \'Etudes Sci. Publ. Math.},
   number={50},
   date={1979},
   pages={27--58},
}
		
\bib{RS96}{book}{
AUTHOR = {Runst, Thomas}
author={Sickel, Winfried},
     TITLE = {Sobolev spaces of fractional order, {N}emytskij operators, and
              nonlinear partial differential equations},
    SERIES = {de Gruyter Series in Nonlinear Analysis and Applications},
    VOLUME = {3},
 PUBLISHER = {Walter de Gruyter \& Co.},
   ADDRESS = {Berlin},
      YEAR = {1996},
     PAGES = {x+547},      
}
		
\bib{S99}{article}{
   author={Starkov, A. N.},
   title={The first cohomology group, mixing, and minimal sets of the
   commutative group of algebraic actions on a torus},
   note={Dynamical systems. 7},
   journal={J. Math. Sci. (New York)},
   volume={95},
   date={1999},
   number={5},
   pages={2576--2582},
}

\bib{T78}{book}{
AUTHOR = {Triebel, Hans},
     TITLE = {Interpolation theory, function spaces, differential operators},
    SERIES = {North-Holland Mathematical Library},
    VOLUME = {18},
 PUBLISHER = {North-Holland Publishing Co.},
   ADDRESS = {Amsterdam},
      YEAR = {1978},
     PAGES = {528},
      ISBN = {0-7204-0710-9},
}
		


\bib{W70}{article}{
   author={Walters, Peter},
   title={Conjugacy properties of affine transformations of nilmanifolds},
   journal={Math. Systems Theory},
   volume={4},
   date={1970},
   pages={327--333},
}

\end{biblist}
\end{bibdiv}
\end{document}